\documentclass[12pt,a4paper]{amsart}
\usepackage[english]{babel}
\usepackage{color}
\usepackage[T1]{fontenc}
\usepackage{amsmath}
\usepackage{amscd}
\usepackage{amstext}
\usepackage{amsbsy}
\usepackage{amsopn}
\usepackage{amsthm}
\usepackage{mathabx}
\usepackage{amsxtra}
\usepackage{upref}
\usepackage{amsfonts}
\usepackage{amssymb}
\usepackage{mathrsfs}
\usepackage{euscript}
\usepackage{array}
\usepackage{wasysym}
\usepackage{stmaryrd}
\usepackage{verbatim}

\theoremstyle{plain}
      \newtheorem{theorem}{Theorem}
      \newtheorem{lemma}[theorem]{Lemma}
      \newtheorem{corollary}[theorem]{Corollary}
      \newtheorem{proposition}[theorem]{Proposition}
     
      \theoremstyle{definition}

     \theoremstyle{remark}
     \newtheorem{remark}[theorem]{Remark}
     \theoremstyle{Fact}
     
     \theoremstyle{Example}
     \newtheorem{exemple}[theorem]{Example}
\theoremstyle{notation}

\newcounter{Step}
\newcommand{\Lim}{\operatorname{Lim}}

\def\R{\mbox{I\hspace{-.15em}R} }
\def\Q{\mbox{l\hspace{-.47em}Q} }
\def\C{\hspace{.17em}\mbox{l\hspace{-.47em}C} }

\def\o{\otimes}

\begin{document}
\title[Time Reversal of free diffusions I]{Time Reversal of free diffusions I :\\ Reversed Brownian motion, Reversed SDE and first order regularity of conjugate variables}
\begin{abstract}
We show that solutions of free stochastic differential equations with regular drifts and diffusion coefficients, when considered backwards in time, still satisfy free SDEs for an explicit Brownian motion and drift. We also study the stochastic integral part with respect to this reversed free Brownian motion of canonical closed martingales. We deduce that conjugate variables computed along a free Brownian motion, an example of such a reversed martingale appearing in the definition of non-microstates free entropy, are in the $L^2$ domain of corresponding free difference quotients for almost every time.   
\end{abstract}


\author[Y. Dabrowski]{Yoann Dabrowski}\address{ 
Universit\'e de Lyon\\
Universit\'e Lyon 1\\
Institut Camille Jordan\\
43 blvd. du 11 novembre 1918\\
F-69622 Villeurbanne cedex\\
France}
\email{dabrowski@math.univ-lyon1.fr}
\thanks{Research partially supported by ANR Grant NEUMANN}

\date{}
\maketitle


\section*{Introduction}

In a fundamental series of papers, Voiculescu introduced analogs of entropy and Fisher information in the context of free probability theory. A first microstate free entropy $\chi(X_{1},...,X_{n})$ is defined as a normalized limit of the volume of sets of microstate i.e. matricial approximations (in moments) of the n-tuple of self-adjoints $X_{i}$ living in a (tracial) $W^{*}$-probability space $M$. Starting from a definition of a free Fisher information $\Phi^{*}(X_{1},...,X_{n})$ and computing it along a free Brownian motion, Voiculescu  \cite{Vo5} also defined a non-microstate free entropy $\chi^{*}(X_{1},...,X_{n})$, known by the fundamental work \cite{BGC} to be greater than the previous microstate entropy, and believed to be equal (at least modulo Connes' embedding conjecture). For more details, we refer the reader to the survey \cite{VoS} for a list of properties as well as applications of free entropies in the theory of von Neumann algebras.

Recently, the study of stationary free Stochastic Differential Equations (free SDEs) enabled progresses in either relating two  quantities derived from free entropy, free entropy dimensions $\delta_0(X_{1},...,X_{n})$ and $\delta^*(X_{1},...,X_{n})$ in the microstates and non-microstates pictures \cite{S07,Dab09}, or in obtaining direct applications of  finite free Fisher information to von Neumann algebras \cite{ID12}. The typical SDE solved in these approaches, first considered in \cite{S07} is
$$Y_{t}^{(i)}=Y_{0}^{(i)}-\frac{1}{2}\int_{0}^{t}\xi_{s}^{(i)}ds + S_{t}^{(i)}$$ where $\xi_{s}^{ (i)}$ is the i-th conjugate variable of $Y_{s}^{(i)}$'s in the sense of \cite{Vo5} (this is the free analogue of the score function), $S_{t}^{(i)}$ a free Brownian motion free with respect to $Y_{0}^{(i)}$. The first general problem is that very few cases where $\xi_{s}^{(i)}$ are Lipschitz are known, and the only known general solutions \cite{Dab10} are thus weak solutions, non-adapted to free Brownian motions. They are thus difficult to use in applications, for instance in \cite{ID12}, having strong solutions, or even a solution built in a free product is absolutely crucial to get the stronger applications~: absence of Cartan subalgebras.

The good property of having stationary solutions, enabling to build trace preserving homomorphism is thus payed by the price of very few examples known to be  well-behaved (since at this point a free version of recent results \cite{KryLp,FlanFed10} of strong solutions under weak regularity seams out of reach).  On the other hand, the definition of free entropy uses only $X_t^{(i)}=X_{0}^{(i)}+ S_{t}^{(i)}$ the simplest process, and instead of trying to modify the process considered, to make it reversible, it seems natural to try proving some reversibility of this given process. 

This is the spirit of the study of time reversal of diffusions in the classical case. Before summarizing the classical literature, one can already say that one expects (and we will prove later) that one needs to introduce the same conjugate variable to compensate the evolution of Brownian motion backwards in time. Precisely, if $\overline{X}_t^{(i)}=X_{T-t}^{(i)}$ and $\overline{\xi}_t^{(i)}$ the conjugate variable computed in this variable, we will prove (propositions \ref{RevSDE} and \ref{BrownianC}) there is a free Brownian motion adapted to a natural filtration $\overline{S}_{t}^{(i)}$ such that on $[0,T]$ :
$$\overline{X}_t^{(i)}=\overline{X}_0^{(i)}-\int_{0}^{t}\overline{\xi}_{s}^{(i)}ds + \overline{S}_{t}^{(i)}.$$

Since it is known from \cite{Vo5} that $\overline{\xi}_t^{(i)}$ is a martingale in the reversed filtration, the main application of the first paper in this series will be to the regularity of $\overline{\xi}_t^{(i)}$  deduced from computing the covariation of this martingale with stochastic integrals for $\overline{S}_{t}^{(i)}$. This will enable us to prove (see Corollary \ref{BrownianConj}) that for almost all $t$ , $\overline{\xi}_t^{(i)}\in D(\overline{\partial_t})$ the $L^2$ closure of the corresponding free difference quotient. We will give more applications to free entropy in the spirit of \cite{ID12} in the second part of this series.

Finally, since we are interested in more general processes than free Brownian motions (for instance the free liberation process of \cite{Vo6} is also crucial for free entropy applications and studied here in propositions \ref{liberationC} and corollary \ref{LiberationConj}), we will find rather general conditions to get an explicit SDE for the reversed process (cf. proposition \ref{RevSDE} and assumption (C) in section 2.1). 

Let us now explain the relation with classical works before summarizing the content of this paper.

Consider a solution on $[0,T]$ of a classical Markovian SDE driven by a classical Brownian motion $B_s$ :
$$X_{t}=X_{0}+ \int_0^tb(s,X(s))ds +\int_0^t\sigma(s,X(s))dB_s.$$

The problem of describing when $Y_t=X_{T-t}$ is not only a Markov process but also a diffusion solving the same kind of  SDEs was first raised by Nelson. 
His motivation came to model Quantum Mechanics (which is reversible) in Stochastic Mechanics. \cite{Nelson67,Nelson} found that formally, there should be a correction of the drift by appropriate score function, i.e. $Y_t$ should satisfy :
$$Y_{t}=Y_{0}+ \int_0^t\overline{b}(T-s,Y(s)))ds +\int_0^t\sigma(T-s,X(s))d\overline{B}_s,$$
with the new drift :
$$\overline{b}_j(T-s,y)=\frac{\sum_i\nabla_i((\sigma\sigma^*)_{ji}p_s)}{p_s}(y)-b_j(T-s,y),$$

where $p_t$ is the density with respect to Lebesgue measure of $(Y_{1,t},...,Y_{n,t}).$

This problem was solved mathematically after a series of contributions \cite{Anderson82,Follmer86,HaussmannPardoux,Pardoux} and the culmination is due to \cite{MNS89} where the reversed process is proved to satisfy the above reversed equation in the sense of a martingale problem as soon as the reversed drift makes sense  with the integrability required in order to formulate the martingale problem (see also \cite{JacodP} for later exploration in a more general L\'evy process context)

More interesting for us, in an earlier work, \cite{Pardoux} gave under stronger conditions an explicit formula for the Brownian motion driving the time reversed process~:
$$\overline{B}_t=B_{T-t}-B_T-\int_{T-t}^T\frac{\sum_i\nabla_i(\sigma_ip_s)}{p_s}(X_s)ds.$$
Even though he used there enlargement of filtration methods, one can also check this is a Brownian motion by a Levy Theorem known in the free case. This will be our approach. We will take the right assumptions in order to have a formula for the reversed Brownian motion, and then check it is indeed one by a free Levy characterization of free Brownian motion, the first version dating back to \cite{BGC}.

Let us finally summarize the content of this paper.

Section 1 gathers preliminaries about free stochastic integration and free SDEs. Since the reversed process is in general not known to be adapted to the filtration of the reversed Brownian motion, we will especially need Ito calculus beyond \cite{BS98} in this non-adapted context, even with bad continuity properties of the filtration. \cite{JungePerrin} is an inspiration there and later where non-commutative continuous time martingale techniques are needed. Section 1.2 will thus clarify this question of Ito formulas. Section 1.3 will fix our setting for SDEs and specify the stronger assumptions we need especially for the forward process. A straightforward time dependent free Ito formula necessary to use PDE techniques is written in section 1.4. Section 1.5 recalls basics on free entropy, appearing especially through properties of conjugate variables, to check our theory applies in the most basic examples.

Section 2 contains the main results about the construction of the reversed process.  We first check the expected process to be a reversed free Brownian motion in section 2.1. Our main assumption (C) used there contains either general properties for (variants of) conjugate variables along our starting diffusion process and the assumption we can solve some PDE that will be crucial to obtain martingales in the reversed filtration. This assumption does not always require the forward process to be a strong solution, but almost. In our basic examples, all those assumptions will be straightforward to check. 
These assumptions are modeled on what we can get in these  examples and in SDEs with regular coefficients we will study elsewhere.

Section 2.2  is rather technical, it summarizes some Dirichlet form techniques we will need, and prove the relation between forward and backward stochastic integrals in the case of regular biadapted integrands. Combined together, this enables us in section 2.3 to prove this relation in full generality in lemma \ref{RevStoInt} and then conclude that the reversed process satisfies the right SDE in proposition \ref{RevSDE}. The Dirichlet form techniques are especially important to make valid the reversed SDE until time $T$ (near time $0$ in the original sense of time) where the drift of the reversed process becomes typically more singular.

Section 3 deals with our first applications to free Brownian motion and liberation processes. Section 3.1 checks our assumption (C) in these cases, mostly gathering known results. Section 3.2 study the regularity of closed reversed martingales starting at the algebra $W^*(X_t)$ of a single time, especially its covariation with stochastic integrals with respect to the reversed free Brownian motion. Here again, our non-commutative Dirichlet form preliminaries are crucial to obtain regularity of the projection of these martingales on the space of stochastic integrals. This is what enables us to prove the stated regularity application to conjugate variables computed on free Brownian motion in section 3.3 Finally, section 3.4 gives an example of computation motivated by a question asked in free probability and demonstrating where the reversed SDE is useful in concrete applications. This is in the case of liberation of projections as in \cite{CollinsKemp}.

\section{Preliminaries}
\subsection{Terminology} 
We work with {\it  tracial} von Neumann algebras $(M,\tau)$, i.e. von Neumann algebras $M$  endowed with a faithful, normal, tracial state $\tau$. We denote by $\|x\|_2=\tau(x^*x)^{1/2}$  the  2-norm associated to $\tau$ and by $\|x\|$ the operator norm.
 We denote by $\mathcal Z(M)$ the {\it center} of $M$, by $\mathcal U(M)$ the {\it group of unitaries} of $M$ and by $(M)_1=\{x\in M|\hskip 0.02in\|x\|\leqslant 1\}$ the {\it unit ball} of $M$. We usually assume that $M$ is {\it separable}.



If $M$ and $N$ are tracial von Neumann algebras, then an $M$-$N$ {\it bimodule} is a Hilbert space $\mathcal H$ endowed with commuting normal $*$-homomorphisms $\pi:M\rightarrow \mathbb B(\mathcal H)$ and $\rho:N^{op}\rightarrow \mathbb B(\mathcal H)$. For $x\in M,y\in N$ and $\xi\in\mathcal H$ we denote $x\xi y=\pi(x)\rho(y)(\xi)$. If $M,N,P$ are tracial von Neumann algebras, $\mathcal H$ and $\mathcal K$ be $M$-$N$ and $N$-$P$ bimodules, respectively, then $\mathcal H{\otimes}_N\mathcal K$ denotes the {\it Connes tensor product}  endowed with the natural $M$-$P$ bimodule structure (see \cite{Po86}).

For a tracial *-algebra $M$, we will write $\mathcal{H}(M,\eta)$ the canonical $M-M$ bimodule associated to the completely positive map $\eta:M\to M$ (see e.g. \cite[section 1.1.2]{Po01}, when $M$ is a von Neumann algebra, $\eta$ will be assumed normal), i.e. the separation completion of $M\otimes M$ for the scalar product $\langle x\o y,x'\o y'\rangle=\tau(y^*\eta(x^*x')y').$ We will still call $1\o 1$ the image of this tensor in $\mathcal{H}(M,\eta)$.



\subsection{Free Brownian motion and stochastic integrals}
We recall here what is a free Brownian motion relative to a
subalgebra $B$ and a completely positive map $\eta:B\to B$ in our finite von Neumann algebra $(M,\tau).$

Thus $B$ is a fixed von Neumann sub-algebra with its canonical $\tau$-preserving conditional expectation $E_B$, $\eta:B\rightarrow B$ a unital completely positive map, assumed to be $\tau$-symmetric ($\tau(\eta(x)y)=\tau(x\eta(y)$) so that via Proposition 2.20 in \cite{SAVal99} the associated $B$-semicircular system is tracial. Given $B_{s}$ be an increasing filtration of von Neumann algebras $B\subset B_{0}$, we will call adapted B-free Brownian motion of covariance $\eta$ a family $S_s^j,j=1...,m,s\geq 0$ of adapted processes such that $(S_s^j)_{s\geq t}$ is Speicher's $B_s$-Gaussian stochastic processes $X_s^j$ with covariance given by $E_{B_{t}}((S_s^i-S_t^i)b(S_u^j-S_t^j))=((s\wedge u)-t)1_{\{i=j\}}\eta(E_{B}(b))$, for any $s,u\geq t$, $b\in B_{t}$. Stated otherwise in the notation of \cite{SAVal99} $W^{*}(B_t, S_s^j,s\geq t)=\Phi(B_t,\tilde{\eta}\circ E_B)$ where $\tilde{\eta}:B\rightarrow B(L^2([t,\infty))\o \C^m)\o B$ given by $\langle 1_{[ts)}\o e_i,\tilde{\eta}(b)(1_{[tu)}\o e_j)\rangle=((s\wedge u)-t)\eta(b)1_{\{i=j\}}$ with $(e_j)$ the canonical basis of $\C^m$. It is well known that in this context $W^*((S_s^j)_{s\geq t})$ is free with amalgamation over $B$ with $B_s$. 

We will need the following free Paul L\'evy's characterization of a free Brownian motion, proved in \cite{BGC} for the case $B=\C$ and in \cite{Dab10} for the general case, using crucially \cite{SpeicherHH}.
\begin{theorem}\label{FreeLevy}
Let $B_{s}$ be an increasing filtration of von Neumann algebras in a non-commutative tracial probability space $(M,\tau)$
$Z_{s}=(Z_{s}^{1},...,Z_{s}^{m}), s\in \R_{+}$ an m-tuple of self-adjoint processes adapted to this filtration $Z_{0}=0$ and~:
\begin{enumerate}
\item $E_{B_{s}}(Z_{t})=Z_{s}$
\item$Z_{t}-Z_{s}=U_{t,s}+V_{t,s}$ with $\tau(|U_{t,s}|^{4})\leq K (t-s)^{3/2}$ and $\tau(|V_{t,s}|^{2})\leq K (t-s)^{2}$ 
\item $\tau(Z_{t}^{k}AZ_{t}^{l}C)=\tau(Z_{s}^{k}AZ_{s}^{l}B)+(t-s)1_{\{k=l\}}\tau(A\eta(E_{B}(C)))+o(t-s)$ for any $A,C\in B_s$.
\end{enumerate}
Then $Z$ is a B-free Brownian motion of covariance $\eta$. 
\end{theorem}
As in \cite{BS98} who wrote the case $B=B_0=\C$, we will consider stochastic integrals $\int_0^TU_s\#dS_s$ defined for $U\in L^2_{ad}([0,T],\mathcal{H}(M,\eta\circ E_B)^m)$ ($ad$ means $U_s\in \mathcal{H}(B_s,\eta\circ E_B)^m)$) and extending linearly isometrically the elementary stochastic integral defined for $U_s=1_{[r,t)}(s)\otimes e_i a_r\otimes b_r$, $a_r,b_r\in B_r$ ($e_i$ meaning the value in the i-th component)
by :

$$\int_0^TU_s\#dS_s:=a_r(S_t^{(i)}-S_r^{(i)})b_r.$$

(Note, as in \cite[p 18]{BGC} that linear combinations of such $U$'s are dense in $L^2_{ad}([0,T],\mathcal{H}(M,\eta\circ E_B)^m$ since they are first dense in continuous functions $C^{0}_{ad}([0,T],\mathcal{H}(M,\eta\circ E_B)^m)$ and those are dense because an adapted process $U_t$ is approximated by $U*\delta_n(s)=\int_0^{1/n}dx \delta_n(x)U_{s-x}$ the adapted time convolution with the continuous functions $\delta_n(x)=n\delta(nx)$ for $\delta$ positive continuous on $[0,1]$ of integral $1$.)

We will need consequences of non-commutative Burkholder-Gundy inequalities (without assumption on the filtration contrary to \cite{BS}). Fix $p\in[2,\infty[$. Let $\mathcal{B}_p^a$ the subspace of $L^2$ adapted processes 
such that $\int_0^T\langle U_s,U_s\rangle ds,\int_0^T\langle U_s^*,U_s^*\rangle ds\in L^{p/2}(M),$ 
with norm 
$$||U||_{\mathcal{B}_p^a}=\max\left(||\int_0^T\langle U_s,U_s\rangle ds||_{p/2},
||\int_0^T\langle U_s^*,U_s^*\rangle ds||_{p/2}\right).$$

Recall that when $U_s=\sum_i e_ia_i\o b_i$, $\langle U_s,U_s\rangle=\sum_ia_i^*\eta(E_B(b_i^*b_i))a_i$ and is then extend by isometry to $\mathcal{H}(M,\eta\circ E_B)^m$ with value in $L^1$ We write $\mathcal{H}_p(M,\eta\circ E_B)^m$ for the subspace where these products $\langle U,U\rangle$, $\langle U^*,U^*\rangle$ are in $L^{p/2}$, with norm $$\|U\|_{\mathcal{H}_p(M,\eta\circ E_B)^m}=\max(\|\langle U,U\rangle\|_{p/2},\|\langle U^*,U^*\rangle\|_{p/2}).$$ Note that $\int_0^T\langle U_s,U_s\rangle ds $ is also extended by isometry to $L^2_{ad}([0,T],\mathcal{H}(M,\eta\circ E_B)^m).$

We will need the following standard fact :

\begin{lemma}\label{simple}
Simple processes are  dense in $\mathcal{B}_p^a,$ $p\in[2,\infty[.$
\end{lemma}
\begin{proof}
As above for $L^2$, $U*\delta_n$ is in $C^{0}_{ad}([0,T],\mathcal{H}_p(M,\eta\circ E_B)^m)$. Note that in all our formulas, $U$ may be extended to $0$ for negative time and time larger than $T$. Indeed, adaptedness is obvious by choice of the support of $\delta$, and we can show continuity for non-adapted processes. Take $a,b\in M$  and bound using sesquilinearity and Cauchy-Schwarz : 
\begin{align*}&|\tau(a^*\langle\int_0^{1/n}dx\delta_n(x)U_{y-x}-\int_0^{1/n}dx\delta_n(x)U_{y'-x},\int_0^{1/n}dz\delta_n(z)U_{t-z}\rangle b)|
\\&\leq(\int_{\min(y,y')-1/n}^{\max(y,y')}dx(\delta_n(y-x)-\delta_n(y'-x)^2)^{1/2} (\int_{t-1/n}^tdx\delta_n(t-x)^2)^{1/2}\times \\&\ \ \ \ \ \  \ \ \ \ \ \  \ (\int_{y-1/n}^ydx\langle ||U_xa||_2^2)^{1/2}(\int_{t-1/n}^tdz\langle ||U_zb||_2^2)^{1/2}\end{align*}

Let $q$ with $2/p+2/q=1$ then taking a sup over $a,b\in L^{q}$ of norm $1$ one gets 
\begin{align*}&||\langle U*\delta_n(y)-U*\delta_n(y'),U*\delta_n(t)\rangle||_{p/2}\leq (\int_{\min(y,y')-1/n}^{\max(y,y')}dx(\delta_n(y-x)-\delta_n(y'-x)^2)^{1/2} \times\\ &(\int_{t-1/n}^tdx\delta_n(t-x)^2)^{1/2}||\int_{y-1/n}^ydx\langle U_s,U_s\rangle ds||_{p/2}^{1/2}||\int_{t-1/n}^tdx\langle U_s,U_s\rangle ds||_{p/2}^{1/2}.\end{align*}
Especially we got by duality (in the case y' really negative so that $U*\delta_n(y')=0$, and $t=y$) that the convolution is indeed in $\mathcal{H}_p(M,\eta\circ E_B)^m,$ and using also symmetric identities and continuity of $\delta$ one also gets continuity of the convolution.
Likewise, using $\int_0^Tdt ||U*\delta_n(t)a||_2^2\leq \int_0^Tdt\left(\int_{\R} \delta_n(y-x)||U_xa||_2\right)^2\leq\left( \int_{\R} \delta_n(y-x)\right)^2 \int_0^Tdt||U_xa||_2^2=\int_0^Tdt||U_xa||_2^2$ using a convolution inequality for real valued functions and the assumption on $\delta$, we get : $$||\int_0^Tdt\langle U*\delta_n(t),U*\delta_n(t)\rangle||_{p/2}\leq ||\int_{0}^Tdx\langle U_s,U_s\rangle ds||_{p/2}.$$

Since continuous process are easy to approximate by simple processes, it suffices to check  $U*\delta_n\to U$ in $\mathcal{B}^a_p.$
But by the continuity above of $U*\delta_n$ in $U$, uniformly in n, it suffices to check it on a dense space of $U$ of non-adapted processes (and do the same with the adjoint of $U$). This will be easier than the adapted case since we can use functional calculus. Verifying convergence is easy on the space of continuous (non adapted) processes since $U*\delta_n(s)-U(s)=\int_0^{1/n}dx\delta_n(x)(U_{s-x}-U_s)=\int_0^{1/n}dx(\tau_{-x}U-U)(s)$ and since $||(\tau_{-x}U-U)||_{B_{p}}\to 0$ for $U$ continuous. To prove density of continuous non-adapted processes the case $p=2$ by Hilbertian techniques (projection on finite dimensional space and reduction to the finite dimensional value $L^2$ space.

Thus let $U\in \mathcal{B}^a_p$ and $U_n \to U$  in $\mathcal{B}^a_2$ with $U_n$ continuous.

For convenience we use the following embedding.

Consider $N=M*_BW(S,B)$ where $S$ is a semicircular of covariance $\eta$

$U\mapsto (s\mapsto U_s\#S)=V(U)$ is clearly an $L^2$ isometry to $L^2([0,T],L^2(N))$ Let $\mathcal{E}$ the conditional expectation of $L^{\infty}([0,T],N)$ to $M$
Then $\int_0^T\langle U_s,U_s\rangle ds=\mathcal{E}(V(U)^*V(U))$ so that $V(U)\in L_p^c(N,\mathcal{E})$ (see e.g. \cite[section 1.3]{JungePerrin} for the main property of this $L^p$ $M$-module)

Thus let  $X_{n,\alpha}=V(U_n(\alpha/(\alpha+\mathcal{E}(V(U_n)^*V(U_n))))^{1/2})$ so that $||\mathcal{E}(X_{n,\alpha}^*X_{n,\alpha})||\leq \alpha$ so that one deduces from convergence to $X_{\alpha}=V(U(\alpha/(\alpha+\mathcal{E}(V(U)^*V(U))))^{1/2})$ in $L_2^c(N,\mathcal{E})$, which is easy using  $\mathcal{E}(V(U_n)^*V(U_n))\to \mathcal{E}(V(U)^*V(U))$ in $L^1$,  the convergence of $X_{n,\alpha}$ in $L_p^c(N,\mathcal{E})$ when $n\to \infty$ by interpolation. The continuity of $V^{-1}(X_{n,\alpha})$ from the one of $U_n$ is also easy. Thus the limit $V^{-1}(X_{\alpha})$ is in the closure of continuous functions, and letting $\alpha\to \infty$ one gets the density result.
\end{proof}
Let us now deduce the consequence for stochastic integrals :
\begin{lemma}\label{Bpa}
For any $U\in \mathcal{B}_p^a, $ we have $\int_0^TU_s\#dS_s\in L^p(M).$
\end{lemma}
\begin{proof}
As in \cite[section 4.2]{BS98} one deduces $\int_0^TU_s\#dS_s\in L^p$ first for simple processes. 
Indeed, in order to prove from Pisier-Xu  non-commutative Burkholder-Gundy inequality that for any simple process, we have :
\begin{equation}\label{PisierXu}c_p||U||_{\mathcal{B}_p^a}\leq ||\int_0^TU_s\#dS_s||_p
\leq C_p||U||_{\mathcal{B}_p^a},\end{equation}
it suffices to check that quadratic variations of simple processes converge in $M$ to $\int_0^T\langle U_s^*,U_s^*\rangle ds$. As in the proof of Ito's formula in \cite{BS98} we first detail the convergence in $M$ for $a,b,c,d\in \mathcal{B}_s$ of $$\sum_{k=1}^na(S_{s+\frac{k+1}{n}(t-s)}-S_{s+\frac{k}{n}(t-s)}^{(i)})bc(S_{s+\frac{k+1}{n}(t-s)}-S_{s+\frac{k}{n}(t-s)}^{(j)})d\to (t-s)a\eta E_B(bc)d\delta_{\{i=j\}}.$$

We consider only the case $i=j$ and remove the indexes.
Since for any free Brownian motion (even if the filtration is not Brownian) $S_{s+\frac{k+1}{n}(t-s)}-S_{s+\frac{k}{n}(t-s)}$ and $\mathcal{B}_s$ are free with amalgamation over $B$, $u_k=(S_{s+\frac{k+1}{n}(t-s)}-S_{s+\frac{k}{n}(t-s)})(bc- E_B(bc))(S_{s+\frac{k+1}{n}(t-s)}-S_{s+\frac{k}{n}(t-s)})$ are free with amalgamation over $B_s$ and so are $v_k=(S_{s+\frac{k+1}{n}(t-s)}-S_{s+\frac{k}{n}(t-s)})E_B(bc)(S_{s+\frac{k+1}{n}(t-s)}-S_{s+\frac{k}{n}(t-s)})-\eta(E_B(bc))(t-s)/n).$ Since they have null conditional expectation on $B_s$ either by freeness with amalgamation or definition of the covariance Thus one can apply free Rosenthal inequality \cite[Proposition 7.1]{Junge05} with $E_{B_s}(u_k^*u_k)=\eta E_B[(bc- E_B(bc))\eta(1)(bc- E_B(bc))](t-s)^2/n^2$ and get :
$$||\sum_k (S_{s+\frac{k+1}{n}(t-s)}-S_{s+\frac{k}{n}(t-s)})(bc- E_B(bc))(S_{s+\frac{k+1}{n}(t-s)}-S_{s+\frac{k}{n}(t-s)})||\leq (4\frac{t-s}{n}+2\frac{t-s}{\sqrt{n}}).$$
The second bound is again as in \cite{BS98} using this time $E_{B_s}(v_k^*v_k)=\eta E_B[(bc- E_B(bc))\eta(1)(bc- E_B(bc))](t-s)^2/n^2$:
$$||\sum_k v_k||\leq ||E_B(bc)||(5\frac{t-s}{n}+2\frac{t-s}{\sqrt{n}}).$$
One deduces the desired convergence. 



From the density of simple processes and inequality \eqref{PisierXu}, stochastic integral extends to $\mathcal{B}_p^a$ and it has to agree with the case $p=2.$
\end{proof}
We now write Ito formula in this context :
\begin{proposition}\label{ItoLp}
With the notation above, take $U\in \mathcal{B}_p^a, W\in \mathcal{B}_q^a$, with $1/p+1/q=1/2$. Then If $X_s=\int_0^sU_v\#dS_v, Y_s\in \int_0^sW_u\#dS_u,$ $UY,XW\in \mathcal{B}_2^a$ and :
$$X_sY_s=\int_0^s (U_vY_v+X_vW_v)\#dS_v+\int_0^s\langle U_v,W_v\rangle dv,$$
\end{proposition}
\begin{proof}
$UY,XV$ are clearly adapted. One can compute when $U$ comes from a simple process  $||U_vY_v||_2^2= \tau( \langle U_v,U_v\rangle Y_vY_v^*)\leq \tau( \langle U_v,U_v\rangle Y_vY_v^*+(Y_T-Y_v)(Y_T-Y_v)^*)=\tau( \langle U_v,U_v\rangle Y_TY_T^*)$ where the martingale property is used in the last inequality, so that : $$||UY||_{\mathcal{B}_p^a}^2=\int_0^T||U_vY_v||_2^2dv \leq \tau( \int_0^Tdv\langle U_v,U_v\rangle Y_TY_T^*)\leq ||Y_T||_{q}^2||\int_0^Tdv\langle U_v,U_v\rangle||_{p/2}\leq C_q^2 ||U||_{\mathcal{B}_p^a}^2||V||_{\mathcal{B}_q^a}^2.$$

This inequality thus extends to the case where $U$ is not a simple process, and thus $UY\in \mathcal{B}_2^a$, the case of $XW$ s similar. Since from \cite[Proposition 2.2]{JungeLp} using the notation of our proof of lemma \ref{simple}, one can write for some $T\in M$, $||T||\leq 1$, $\int_0^s\langle U_v,W_v\rangle dv$  as:

\noindent $\mathcal{E}(V(U1_{[0,s]})^*V(W1_{[0,s]}))=\mathcal{E}(V(U1_{[0,s]})^*V(U1_{[0,s]}))^{1/2}T\mathcal{E}(V(W1_{[0,s]})^*V(W1_{[0,s]}))^{1/2}$ and one deduces from Hölder inequality : $$ ||\int_0^s\langle U_v,W_v\rangle dv||_2\leq ||U||_{\mathcal{B}_p^a}||V||_{\mathcal{B}_q^a}.$$

Thus by continuity of both sides of the identity for the norms $\mathcal{B}_p^a,\mathcal{B}_q^a,$ it suffices to get the simple process case, and the proof is similar to the one in \cite{BS98}, since we have already computed the quadratic variation in our previous lemma.

\end{proof}

We will also consider $\mathcal{B}_\infty^a$ (note the difference with \cite{BS98}) the completion of adapted simple processes for the norm $$||U||_{\mathcal{B}_\infty^a}^2:=\int_0^T||U_s||_{M\hat{\o}M}^2ds,$$ with $(a\o b)^*=b^*\o a^*.$ On this space the stochastic integral is expected to extend to continuous map valued in $M$, but this requires  at this stage supplementary assumptions as follows :

\begin{lemma}
Assume either $B=\C$ or the filtration is the Brownian filtration i.e. $B_s=W^*(S_t,t<s,B_0)$, then for any simple adapted process $U$  valued in $M\otimes M$:

\noindent $||\int_0^TU_s\#dS_s||\leq 3||U||_{\mathcal{B}_\infty^a}.$
Thus the stochastic integral extends to a bounded map $\mathcal{B}_\infty^a\to M.$
\end{lemma}

\begin{proof}
It is proved in \cite[Corollary 1.10]{JRS} that  in the second case :$$||\int_0^TU_s\#dS_s||\leq||\int_0^Tm\circ (1\otimes (\eta\circ E_B\circ m)\o 1)(U_s^*\otimes U_s)ds||^{1/2}+||\int_0^Tm\circ (1\otimes (\eta\circ E_B\circ m)\o 1)(U_s\otimes U_s^*)ds||^{1/2}.$$

We have moreover the following concluding inequality : \begin{align*}||m\circ (1\otimes (\eta\circ E_B\circ m)\o 1)(U_s^*\otimes U_s)||&=\sup_{||\xi||_2\leq 1}\langle \xi,m\circ (1\otimes (\eta\circ E_B\circ m)\o 1)(U_s^*\otimes U_s)\xi\rangle \\&=\sup_{||\xi||_2\leq 1}|| U_s1\otimes \xi\|_{\mathcal{H}(M,\eta\circ E_B)}\leq ||U_s||_{M\hat{\o}M}^2.\end{align*}

In the case $B=\C$, one can follow the proof of \cite[Th 3.2.1]{BS} without assuming a strong assumption on the filtration using only the property of the coarse correspondence. In this case we actually prove $$||\int_0^TU_s\#dS_s||\leq 2\sqrt{2}\int_0^T||U_s||_{M\overline{\o}M^{op}}^2ds.$$
\end{proof}

From this point, one can easily adapt the proof of Ito Formula from \cite{BS} to this context (using \cite[Proposition 7.1]{Junge05} instead of \cite{Vo87}), we only state the result and leave the proof to the reader.

\begin{proposition}\label{ItoSimple}
Assume either $B=\C$ or the filtration is the Brownian filtration. Take $U\in L^2_{ad}([0,T],\mathcal{H}(M,\eta\circ E_B)^m)$ $V\in \mathcal{B}_\infty^a$. Then If $X_s=\int_0^sU_v\#dS_v, Y_s\in \int_0^sV_u\#dS_u,$ $UY,YU,XV,VX\in L^2_{ad}([0,T],\mathcal{H}(M,\eta\circ E_B)^m)$ and :

$$X_sY_s=\int_0^s (U_vY_v+X_vV_v)\#dS_v+\int_0^sm\circ (1\otimes (\eta\circ E_B\circ m)\o 1)(U_v\otimes V_v)dv,$$
$$Y_sX_s=\int_0^s (Y_vU_v+V_vX_v)\#dS_v+ +\int_0^Tm\circ (1\otimes (\eta\circ E_B\circ m)\o 1)(V_v\otimes U_v)dv.$$
Moreover, the quadratic variation converges in $L^2$ :
$$\sum_{k=1}^n(X_{(k+1)s/n}-X_{ks/n})(Y_{(k+1)s/n}-Y_{ks/n})\to_{n\to \infty}\int_0^sm\circ (1\otimes (\eta\circ E_B\circ m)\o 1)(U_v\otimes V_v)dv.$$
\end{proposition}

\subsection{Setting for free SDEs and free markovianity}

We will always consider the following setting when we suppose given solutions of free SDEs

\begin{minipage}{15,5cm} \textbf{Assumption (A)}:
Let $B\subset A_1,...,A_n$; $\alpha_{0}(A_1),...,\alpha_{0}(A_n)\subset M$ be *-subalgebras algebraically independent over $B$, a von Neumann subalgebra in $(M,\tau)$, for a $*$-homomorphism $\alpha_0$ injective on $B$. $A=Alg(A_1,...,A_n)$ the (abstract free) algebra (over B) they generate with a canonical extension $\alpha_0$.
We consider given in $(M,\tau)$ p $B$-free Brownian motion $S_t=(S_t^{(1)},...,S_t^{(p)})$ of covariance $\eta$, free from $M$ with amalgamation over $B$ adapted to a filtration $\mathcal{F}_t\supset W^*(\alpha_0(A),S_s,s\leq t)$

\begin{enumerate}
\item Assume given a densely defined real derivation $\delta:A\to (A\otimes  A)^p$ (real means $\delta_i(x^*)=(\delta_i(x))^*$ with $(a\otimes b)^*=b^*\o a^*$), $\delta(B)=0$.
\item We assume given a family of *-homomorphisms $\alpha_t:A\to \mathcal{F}_t, t\in [0,T].$
\item We also fix operators $\Delta_{Q,s}:A\to L^{\infty}(A,\tau\circ\alpha_s)$ such that $\Delta_{Q,s}(x^*)=\Delta_{Q,s}(x)^*$ of the form $\Delta_{Q,s}(x)=\tilde{\delta}_s(x)+\frac{1}{2}\sum_i m\circ 1\otimes \eta \circ E_B\circ\alpha_s\otimes 1(\delta_i\otimes 1\delta_i(x)+1\o\delta_i\delta_i(x)),$ with $\tilde{\delta}=\sum_j\delta_j(x)\#Q_{j,s}:A\to L^{\infty}(A,\tau\circ\alpha_s)$ a derivation, with drift $Q_{j,s}\in L^{\infty}(A,\tau\circ\alpha_s)$ such that $s\mapsto \alpha_s(Q_{j,s})\in L^1([0,T],L^q(M))$ for some $q\in[1,\infty)$ (note also $\tilde{\delta}(B)=0$).
\item We assume without loss of generality $\mathcal{F}_t=W^*(\alpha_s(A),S_s,s\in[0,t]))$ so that again without loss of generality $\alpha_t$ extends to $L^{\infty}(A,\tau\circ\alpha_s)$. Moreover we assume $\alpha_t$ satisfies for any $X\in A:$
$$\alpha_t(X)=\alpha_0(X)+\int_0^tds\alpha_s(\Delta_{Q,s}(X))+\int_0^t\alpha_s\otimes\alpha_s(\delta(X))\#dS_s.$$
\end{enumerate}
\end{minipage}

\begin{minipage}{15,5cm} \textbf{Assumption (B)}:

\begin{enumerate}
\item Assume Assumption (A) with $Q_j=Q_{j,s}\in A$ ( independent of $s$ and so that we can define $\Delta_{Q,s}:A\to A$).
\item We assume either $B=\C$  and the solution is freely markovian, i.e. for any t, $W^*(\alpha_s(A),S_s,s\in[0,T])=W^*(\alpha_s(A),S_s,s\in[0,t])*_{W^*(\alpha_t(A))}W^*(\alpha_s(A),S_s-S_t,s\in[t,T])$ or (if $B\neq \C$) the solution is a strong solution, i.e. for any $t\leq \tau\leq T$ : $W^*(\alpha_s(A),S_s-S_t,s\in[t,\tau])=W^*(\alpha_t(A),S_s-S_t,s\in[t,\tau])$ so that  $\mathcal{F}_t= W^*(A,S_s,s\leq t).$
\end{enumerate}
\end{minipage}

An application of Ito's formula 
 shows that the generator is compatible with the *-homomorphism property.

\begin{exemple}\label{BrownianDiffusion} We have a first typical example in free probability where $A=B\langle X_1,...,X_n\rangle$ (i.e. $A_i=B\langle X_i\rangle$) $\delta=(\partial_1,...,\partial_n)$ is the free difference quotient such that $\partial_i(X_j)=1\o 1_{\{i=j\}}$. In this case, the equation in assumption (A) can be written :
$$\alpha_t(X_i):=X_{i,t}=\alpha_0(X_i)+\int_0^tdsQ_i(X_{s})+S_t^{(i)}.$$
A strong solution can be obtained for small time as in \cite{BS} at least for $T$ small enough.
\end{exemple}
\begin{exemple}\label{liberation}We have a second typical example in free probability for $A_i=Alg(B,B_i)$ (whithout loss of generality $B_i$ is an algebra) as above $\delta=(\delta_1,...,\delta_n)$ is the  derivation such that $\delta_k(a_j)=-i(a_j\o 1_{\{k=j\}}- 1_{\{k=j\}}\o a_j),a_j\in B_j, \delta_k(b)=0,b\in B$($i=\sqrt{-1}$)
. In this case, the equation of assumption (A.4) can be written for $a_j\in B_j$:
$$\alpha_t(a_j)=\alpha_0(a_j)+\int_0^tds\alpha_s(i[Q_j,a_j]-a_j+\eta(E_B(\alpha_s(a_j))))+i\int_0^t[dS_s^{(j)},\alpha_s(a_j)].$$

In this case, when $Q_j=0$, the  strong solution (necessarily unique by linearity of the equation in this case) can be obtained using free unitary Brownian motions, i.e. solution $U_t^{j}$ of the linear SDE :
$$U_t^{j}=1-\frac{1}{2}\int _0^tU_s^{j}ds+i\int_0^tdS_s^{(j)}U_s^j.$$
An easy application of Ito formula gives $\alpha_t(a_j)=U_t^{j}\alpha_0(a_j)U_t^{j*}, a_j\in B_j$ is a strong solution which extends to $*$-homomorphism solution on $A$ with $\alpha_t(b)=\alpha_0(b),b\in B$
\end{exemple}

For consistency purposes, we check that the strong solution assumption is stronger than the free markovianity assumption made when $B=\C$ :

\begin{proposition}\label{StrongSolution}
In the setting of assumption (A), a strong solution is always freely markovian and its filtration is continuous.
\end{proposition}
 \begin{proof}By assumption we have $\mathcal{F}_t=W^*(\alpha_s(A),S_s,s\in[0,t])= W^*(A,S_s,s\leq t)$ and  $W^*(\alpha_s(A),S_s-S_t,s\in[t,T])=W^*(\alpha_t(A),S_s-S_t,s\in[t,T])$. We have thus to show
 $\mathcal{F}_T=\mathcal{F}_t*_{W^*(\alpha_t(A))}W^*(\alpha_t(A),S_s-S_t,s\in[t,T]),$ and the continuity.  
 
 But by definition of a free Brownian motion  for $t\leq u\leq T$
 $\mathcal{F}_u=\mathcal{F}_t*_{W^*(B)}W^*(B,S_s-S_t,s\in[t,u]),$
so that by an explicit description of projections in $L^2$ for free products $\cap_{t<u<T} \mathcal{F}_u=\mathcal{F}_t*_{W^*(B)}\cap_{t<u<T}W^*(B,S_s-S_t,s\in[t,u])=\mathcal{F}_t$. The last equality comes from $\cap_{t<u<T}W^*(S_s-S_t,s\in[t,u])$ free with amalgamation over $B$ with $W^*(B,S_s-S_t,s\in[t,T])$ in this same algebra, thus it equals $W^*(B),$ by faithfulness. This gives the right continuity thus obviously the continuity of the filtration.
  
The equality above defining the strong solution property of course especially also implies 
$W^*(\alpha_t(A),S_s-S_t,s\in[t,T])=W^*(\alpha_t(A))*_{W^*(B)}W^*(B,S_s-S_t,s\in[t,T])$
so that by general associativity properties of freeness with amalgamation 
 $\mathcal{F}_t*_{W^*(\alpha_t(A))}W^*(\alpha_t(A),S_s-S_t,s\in[t,T])=\mathcal{F}_t*_{W^*(\alpha_t(A))}W^*(\alpha_t(A))*_{W^*(B)}W^*(S_s-S_t,s\in[t,T])=\mathcal{F}_t*_{W^*(B)}W^*(S_s-S_t,s\in[t,T])=\mathcal{F}_T$ as expected.

 \end{proof}

\subsection{A time dependent free Ito formula}\label{timeIto}

In order to use free PDE's dual to free SDE's, we will need a time dependent Ito formula. This is a mere adaptation of well-known results. Let us first fix notation. Consider the setting of assumption (A). 
Note here we don't assume having a strong solution, since the results of this section will be applied to the reversed process for which this is generally unknown.

We want to define $C^1$ maps with value in the algebra $A$ and for this we need a seminorm consistent with all representations above. For $X\in A_i$, we first write :

$$||X||_{A_i}=\sup\{||\alpha_t(X)||_{B(L^2(\alpha_t(A),\tau))}, t\in {[0,T]}\}.$$

From the assumed continuity of $\alpha_t(X)$ (coming from (A.4)), this is obviously a $C^*$-norm.
Let us write $\mathcal{A}_i$ the completion of $A_i$ for this  $C^*$ norm, and then  $\mathcal{A}=\mathcal{A}_1*_B...*_B\mathcal{A}_n$ the universal free product $C^*$ algebra with amalgamation over $B$ (or rather its completion in any of $\mathcal{A}_i$).

We call $C^2(A_1,...,A_n:B)$ the separation completion of the algebraic product $Alg(A_1,...,A_n)$ in the norm $$||X||_{C^2}=||X||_{\mathcal{A}}+\sum_i||\delta_i(X)||_{(\mathcal{A}\hat{\o}\mathcal{A})}+\sum_{i,j}||(\delta_i\o1)\delta_j(X)||_{(\mathcal{A}\hat{\o}\mathcal{A}\hat{\o}\mathcal{A})}+||(1\o\delta_i)\delta_j(X)||_{(\mathcal{A}\hat{\o}\mathcal{A}\hat{\o}\mathcal{A})}.$$

We write $\iota:A\to C^2(A_1,...,A_n:B)$ the canonical (not necessarily injective) map.

In this way, it is easy to see $\delta$ extends to $C^2(A_1,...,A_n:B)\to \mathcal{A}\hat{\o}\mathcal{A}$ and $\Delta_{Q,s}$ extends to a map from 
$C^2(A_1,...,A_n:B)\to L^{\infty}(A,\tau\circ\alpha_s).$
We will also write $\tau_t$ the law on $\mathcal{A}$ coming from $\tau$ restricted to $\alpha_t(A_1),...,\alpha_t(A_n)$
Note that in the case of assumption (B) $\Delta_{Q,s}$ actually depends on time only through $\tau_s$ and extends to a map from 
$\Delta_{Q}^{\tau_s}:C^2(A_1,...,A_n:B)\to \mathcal{A}$. 

We will also consider a space of $C^1$ maps 
$C^1([s_1,s_2],C^2(A_1,...,A_n:B))$, for $s_1\leq s_2\leq T$, defined as the completion of the space of $C^1$ maps into $A$ for the seminorm of $C^2(A_1,...,A_n:B)$. The  seminorm of the completion is given by :

$$||P||_{C^1([s_1,s_2],C^2)}=\sup_{t\in[s_1,s_2]}\left(||P_t||_{C^2}+||\frac{\partial P_t}{\partial t}||_{C^2}
\right).$$

[Note $T$ is fixed and we defined functions on $[s_1,s_2]$ involving a norm $C^2$ involving all the representations on $[0,T]$ for latter notational convenience. Especially, in this way, $C^1([0,s], C^2(A_1,...,A_n:B))]\simeq C^1([T-s,T], C^2(A_1,...,A_n:B))]$ by change of time $s\mapsto T-s$.]
We can now state our Ito formula :

\begin{proposition}\label{TimeIto}
With the notation above assuming assumption (A), consider

\noindent $P\in C^1([0,T],C^2(A_1,...,A_n:B))$, $t\mapsto \alpha_t(P_t)$, then :
$$\alpha_t(P_t)=\alpha_0(P_0)+\int_0^tds[\alpha_s(\Delta_{Q,s}(P_s))+\alpha_s(\frac{\partial P_s}{\partial s})]+\int_0^t\alpha_s\otimes\alpha_s(\delta(P_s))\#dS_s,$$
with the integral making sense in $L^q$ for $q$ as in assumption (A).
\end{proposition}

\begin{proof}

From fundamental Th. of calculus valued in the Banach space $\mathcal{A}$ we have :

$P_t=P_0+\int_0^tds\frac{\partial P_s}{\partial s},$ and we will apply $\alpha_t$ to this equation.

If $\frac{\partial P_s}{\partial s}$ were in $A$, our assumption would give :
$$\alpha_t(\frac{\partial P_s}{\partial s})=\alpha_{s}(\frac{\partial P_s}{\partial s})+\int_s^tdu\alpha_u(\Delta_{Q,u}(\frac{\partial P_s}{\partial s}))+\int_s^t\alpha_u\otimes\alpha_u(\delta(\frac{\partial P_s}{\partial s}))\#dS_u.$$
From the definition as a completion, and all the continuity of the maps, this equation easily extends to the general case we now consider.

Moreover, from the continuity of $\alpha_s$ and since  $\frac{\partial P_s}{\partial s}$ is continuous in $C^2(A_1,...,A_n:B)$, one can easily apply Fubini Theorem :

$$\int_0^tds\int_s^tdu\alpha_u(\Delta_{Q,u}(\frac{\partial P_s}{\partial s}))=\int_0^tdu\alpha_u(\Delta_{Q,u}(\int_0^uds\frac{\partial P_s}{\partial s})),$$
and a stochastic variant (for which by isometry it suffices to apply a Fubini after computing a scalar product with any stochastic integral): 

$$\int_0^tds\int_s^t\alpha_u\otimes\alpha_u(\delta(\frac{\partial P_s}{\partial s}))\#dS_u=\int_0^t\alpha_u\otimes\alpha_u(\delta(\int_0^uds\frac{\partial P_s}{\partial s}))\#dS_u.$$
Combining, our four equations and the equation for $\alpha_t(P_0)$, we easily get the result. 

\end{proof}

\

\subsection{Non-microstates free entropy and mutual information}\label{Entropy}

 Let us recall the definition of free entropy with respect to a completely positive map $\eta:B\to B$ (assumed to be $\tau$-symmetric on $B$, for $\tau$ the given trace on $B$; i.e. $\tau(b\eta(b'))=\tau(\eta(b)b')$) as in \cite{ShlyFreeAmalg00} (with a change of  normalization). This generalizes Voiculescu's Non-microstates free entropy relative to $B$ corresponding to the case $\eta=\tau$.

 We first define relative free Fisher information. Consider $B\subset (M,\tau)$ ,and selfadjoint elements $X_1,...,X_n\in (M,\tau)$. We assume $X_1,...,X_n$ algebraically free with $B$ (maybe modulo using lemma 3.2 in \cite{ShlyFreeAmalg00} after going to a (non-faithful) system with the same joint law)
Then the $i$-th free-difference quotient is defined as the unique derivation $\delta_i:B\langle X_1,...,X_n\rangle\to \mathcal{H}(B\langle X_1,...,X_n\rangle,\eta\circ E_B),$ such that $\delta_i(b)=0$, for $b\in B$, $\delta_i(X_j)=1_{\{i=j\}}\otimes 1.$ 
Looking at $\delta_i$ as an unbounded operator $L^2(B\langle X_1,...,X_n\rangle,\tau)\to \mathcal{H}(B\langle X_1,...,X_n\rangle,\eta\circ E_B)$, the i-th conjugate variable $\xi_i(X_1,...,X_n:B,\eta)$ is defined as $\delta_i^*1\o 1$ if it exists in $L^2(B\langle X_1,...,X_n\rangle,\tau).$

We then define Fisher information as 

$$\Phi^*(X_1,...,X_n:B,\eta):=\sum_{i=1}^n||\xi_i(X_1,...,X_n:B,\eta)||_2^2.$$
 
 if all conjugate variable exists and $+\infty$ otherwise.
 
 Let $S_s=(S_s^{(1)},...,S_s^{(n)})$ a B-free Brownian motion of covariance $\eta$ free with amalgamation over $B$ with $X=(X_1,...,X_n)$, and let $X_t=X+S_t.$
Then from \cite[Proposition 3.12]{ShlyFreeAmalg00}(variant of \cite[Corollary 3.9]{Vo5}) $\xi_i(X_t;B,\eta)$ exits and equals $\xi_i(X_t;B,\eta)=E_{L^2(B\langle X_t\rangle)}[S_t^{(i)}]/t$ and also from \cite[Proposition 3.11]{ShlyFreeAmalg00}(variant of \cite[Proposition 3.7]{Vo5}) $\xi_i(X_t;B,\eta)=E_{L^2(B\langle X_t\rangle)}[\xi_i(X_s;B,\eta)]$ for $0<s\leq t.$
 
 Especially, it is known (see \cite[Proposition 4.9]{ShlyFreeAmalg00}) that $t\mapsto \Phi^*(X_t;B,\eta)$
 is decreasing, right continuous, bounded by
   $$n^2\tau(\eta(1))^2(\sum_i\tau(X_i^2)+n\tau(\eta(1))t)^{-1}\leq \Phi^*(X_t;B,\eta)\leq n\tau(\eta(1))t^{-1}.$$
  One then defines free entropy (with a corrected normalization) :
 
\begin{align*} \chi^{*}(X_{1},...,X_{m};B,\eta)=\frac{1}{2}&\int_0^\infty\left(\frac{n\tau(\eta(1))}{1+t}-\Phi^*(X_t;B,\eta)\right)dt+\frac{m}{2}\log (2\pi e)\tau(\eta(1)),
\end{align*}
We recall the mutual information defined in \cite[section 14.3]{Vo6}.

Consider $B_1,...,B_n,B$ a bunch of (unital *-)algebras algebraically independent so that $A_i=Alg(B_i,B)$ are algebraically free over $B$  in a finite von Neumann algebra $(M,\tau)$,  a completely positive map. We define $\delta_{B_i:B_1\vee ... \hat{B_i}...\vee B_n\vee B}=\delta_i$ as the derivation on $A=A_1\vee...\vee A_n$ (the algebra generated by $A_i$'s)
$\delta_i: A\to A\o A$ such that $\delta_i(a)=a\o1-1\o a$ if $a\in B_i$ and $\delta_i(a)=0$ if $a\in A_k,k\neq i$ or $a\in B$. Seeing $\delta_i :L^2(A,\tau)\to L^2(A,\tau)\o L^2(A,\tau)$ as a densely defined unbounded operator, the liberation gradient is defined as $j(B_i:B_1\vee ... \hat{B_i}...\vee B_n\vee B)=\delta_i^{*}1\o 1\in L^2(A,\tau)$ if it exists.
Then we define the liberation Fisher information $\varphi^*(B_1,...,B_n:)=\sum_{i=1}^n||\delta_i^{*}1\o 1||_2^2$.

Consider $S_t^{(i)}$ a free Brownian motion of  free from $A$. Consider also $U_t^{(i)}$ as in example \ref{liberation}, the corresponding unitary Brownian motion, i.e. the solution of the free S.D.E. $U_t^{(i)}=1-\int_0^t\frac{1}{2}U_s^{(i)}ds+i\int_0^tdS_s^{(i)}U_s^{(i)}$.

Finally mutual information is defined by :
$$i^*(B_1,...,B_n:B)=\frac{1}{2}\int_0^\infty \varphi^*(U_s^{(1)}B_1U_s^{(1)*},...,U_s^{(n)}B_nU_s^{(n)*}:B)\in[0,\infty].$$


We recall key properties of \cite{Vo6}, contained essentially in the proof of proposition 9.10.

\begin{proposition}
\begin{enumerate}
\item[(i)] $j(U_s^{(i)}B_iU_s^{(i)*}:U_s^{(1)}B_1U_s^{(1)*}\vee ... \widehat{U_s^{(i)}B_iU_s^{(i)*}}...\vee U_s^{(n)}B_nU_s^{(n)*}; B)$ exists for any $s>0$.
\item[(ii)] $s\mapsto \varphi^*(U_s^{(1)}B_1U_s^{(1)*},...,U_s^{(n)}B_nU_s^{(n)*}:B)$ is decreasing right continuous  and $s\mapsto j(U_s^{(i)}B_iU_s^{(i)*}:U_s^{(1)}B_1U_s^{(1)*}\vee ... \widehat{U_s^{(i)}B_iU_s^{(i)*}}...\vee U_s^{(n)}B_nU_s^{(n)*}; B)$
is right continuous with left limits and at most countably many  discontinuity points (everything in $L^2$ norm).
\end{enumerate}
\end{proposition}
\begin{proof}


(i) is a consequence of \cite[Proposition 9.4, Remark 8.10]{Vo6}

(ii) From \cite[Proposition 5.14]{Vo6} with $\tau$ replaced by $E_B$ in the conclusion, we deduce as for its consequence proposition 5.13 that \begin{align*}&j(U_s^{(i)}B_iU_s^{(i)*}:U_s^{(1)}B_1U_s^{(1)*}\vee ... \widehat{U_s^{(i)}B_iU_s^{(i)*}}...\vee U_s^{(n)}B_nU_s^{(n)*}; B)\\&=j(U_s^{(i)}B_iU_s^{(i)*}:U_s^{(1)}B_1U_s^{(1)*}\vee ... \widehat{U_s^{(i)}B_iU_s^{(i)*}}...\vee U_s^{(n)}B_nU_s^{(n)*}\vee W^*(U_t^{(j)}U_s^{(j)*},j,t\geq s); B).\end{align*}
Then, as in Proposition 5.1 there again, for $t>s$, call $\delta_i$ the derivation associated with the last conjugate variable, and $\partial_i$ the one associated to the one in the next formula, we see $\partial_i$ coincides with the restriction of $\delta_i\#(U_s^{(i)}U_t^{(i)*}\o U_t^{(i)}U_s^{(i)*})$ so that if $E_t=E_{U_t^{(1)}A_1U_t^{(1)*}\vee ...\vee U_t^{(n)}A_nU_t^{(n)*}\vee B}$
\begin{align}\label{projLiberation}\begin{split}&j(U_t^{(i)}A_iU_t^{(i)*}:U_t^{(1)}A_1U_t^{(1)*}\vee ... \widehat{U_t^{(i)}A_iU_t^{(i)*}}...\vee U_t^{(n)}A_nU_t^{(n)*}; B)\\&=E_t[U_t^{(i)}U_s^{(i)*}j(U_s^{(i)}B_iU_s^{(i)*}:U_s^{(1)}B_1U_s^{(1)*}\vee ... \widehat{U_s^{(i)}B_iU_s^{(i)*}}...\vee U_s^{(n)}B_nU_s^{(n)*}; B)U_s^{(i)}U_t^{(i)*}].\end{split}\end{align}
The stated decreasingness property follows. From free markovianity, $E_t$ above can be replaced by $E_{\geq t}=E_{U_u^{(1)}A_1U_u^{(1)*}\vee ...\vee U_u^{(n)}A_nU_u^{(n)*}\vee B,u\geq t}$ which is right continuous with left limits as the corresponding filtration.  The continuity statement for liberation Fisher information and liberation gradients then follow from this and the continuity in $M$ of $U_t^{(i)*}.$ 


\end{proof}

\section{Construction of the free diffusion for the time reversal}
Let 
$B\subset M_0\subset (M,\tau)$ separable and let $(S_{1,t},...,S_{n,t})\in M$ a free Brownian motion of covariance $\eta$  for $\eta:B\to B$ symmetric ucp map, free from $M_0$
with amalgamation over $B$. 


Let us consider the situation of assumption (B) and thus we assume given a freely markovian solution for $X\in A$ :
$$X_t:=\alpha_t(X)=\alpha_0(X)+\int_0^tds\alpha_s(\Delta_{Q,s}(X))+\int_0^t\alpha_s\otimes\alpha_s(\delta(X))\#dS_s.$$

We want to describe the time reversed process $\overline{X}_{t}=X_{T-t}$ as a semimartingale with respect to the  canonical reversed filtration (for any $X\in A$). Let us consider $\overline{\xi}_{s}=(\overline{\xi}_{1,s},..,\overline{\xi}_{n,s})=(\delta_1^*1\o1,...,\delta_n^*1\o1)$ computed in $W^*(\alpha_{T-s}(A))$

We first recall a lemma essentially due to Voiculescu \cite{Voi5} :

\begin{lemma}\label{VoiculescuFormula}
In the setting of assumption (A.1), and assuming $\overline{\xi}_{s}=(\overline{\xi}_{1,s},..,\overline{\xi}_{n,s})=(\delta_1^*1\o1,...,\delta_n^*1\o1)$ exists in $L^2(W^*(\alpha_{T-s}(A))^n$ as above then $(A\otimes A)^p\subset D(\delta^*)$ and moreover for any $a\o b\in \alpha_{T-s}(A)\otimes \alpha_{T-s}(A)$:
$$\delta_i^*(a\o b)=a\overline{\xi}_{i,s}b-[m\circ(1\otimes \eta\circ E_B)(\delta_i(a))]b-a[m\circ(\eta\circ E_B\otimes 1)(\delta_i(b))].$$
P):=
Especially, $\alpha_{T-s}(A)\subset D(\delta_s^*\delta_s)$ and for any $P\in A$ : $$\Delta_s(P):= \delta_s^*\delta_s(P)=\delta_s(P)\#\overline{\xi}_s-\sum_i[m\circ( 1\otimes \eta\circ E_B\otimes 1)]((\delta_{i,s}\o 1)\delta_{i,s}(P)+(1\o \delta_{i,s})\delta_{i,s}(P)).$$
Finally, for any $Z\in D(\bar{\delta}_s)\cap W^*(\alpha_{T-s}(A))$, there exists a sequence $Z_{n}\in \alpha_{T-s}(A)$ with $||Z_{n}||\leq ||Z||$, $||Z_{n}-Z||_{2},||\delta_s(Z_{n})-\bar{\delta}_s(Z)||_{2}\rightarrow 0$.
\end{lemma}
\begin{proof}
For $Z\in D(\delta_i)$, it suffices to notice using the derivation property and then the realness of the derivation :
\begin{align*}
\langle& \delta_i(Z),a\o b \rangle_{H(W^*(\alpha_{T-s}(A),\eta E_B)}=\langle \delta_i(a^*Zb^*)-\delta_i(a^*)Zb^*-a^*Z\delta_i(b^*),1\o 1 \rangle_{H(W^*(\alpha_{T-s}(A),\eta E_B)}
\\&=\langle Z,a\overline{\xi}_{i,s}b \rangle_{H(W^*(\alpha_{T-s}(A),\eta E_B)}-\tau(Z^*(m\circ(1\otimes \eta E_B)((\delta_i(a^*))^*))b)-\tau(Z^*a(m\circ\eta\circ E_B\otimes 1(\delta_i(b^*))^*).
\\&=\langle Z, a\overline{\xi}_{i,s}b-[m\circ(1\otimes \eta\circ E_B)(\delta_i(a))]b-a[m\circ(\eta\circ E_B\otimes 1)(\delta_i(b))]\rangle
\end{align*}
Recall $(c\otimes d)^*=(d^*\otimes c^*)$ so that we used by traciality  $\tau((a^*Zc)^*\eta E_B(d^*))=\tau(c^*Z^*a \eta E_B(d^*))=\tau[Z^*a(m\circ\eta E_B\o 1)((c\o d)^*)],$ and a symmetric variant. Since $\delta_i$ are valued in $A\o A$ by assumption, the second statement is now obvious. For the last statement see e.g. proposition 6 and remark 7 in 
 \cite{Dab08}.

\end{proof}

\subsection{Reversed free Brownian motion}
We will need the following assumption (essentially coming from \cite{Vo5} and \cite{Vo6} in cases of free Brownian motion and liberation processes as we will see in the next section \ref{ApplicationSec} bellow). To state one of the main assumptions using the solution of a PDE, we use the notation of section \ref{timeIto}. 

\begin{minipage}{15,5cm} \textbf{Assumption (C)}:
\begin{enumerate}
\item[(0)]There is a dense countable $\Q$-subalgebra $\mathscr{A}$ of $A$ such that $\alpha_{T-s}(\mathscr{A})$  is still a core of $\delta_s$ for any $s$ and the solution of our SDE satisfy assumption (B).
\item $s\in[0,T)\mapsto \overline{\xi}_{s}$ is left continuous with right limits when seen as valued in $L^2(M)^n.$ Moreover it is continuous except on countably many points.
\item $\exists C>0,||\overline{\xi}_{i,s}||<\frac{C}{\sqrt{T-s}},s<T$
\item $\exists D\geq 0,\alpha>0\forall t<s<T,\ ||E_{W^*(\overline{X}_{1,t},...,\overline{X}_{i,t})}(\overline{\xi}_{s,i})-\overline{\xi}_{t,i}||_2\leq D(s-t)^\alpha\frac{1}{\sqrt{T-s}}.$ 
\item The filtration $\mathcal{F}_s$ is continuous.
\item For $s\in[0,T)$, the derivation $\delta_s$ defined on $\alpha_{T-s}(A)$ as $\delta_s(\alpha_{T-s}(X))=\alpha_{T-s}\otimes \alpha_{T-s}(\delta(X))$ extends to a densely defined closable derivation $\hat{\delta}_s$ on $\overline{\mathcal{F}}_{s,alg}=Alg(B,\alpha_{T-u}(A),u\leq s, S_{T-u}-S_{T-s})$ in such a way that $\hat{\delta}_s( S_{T-u}-S_{T-s})=0,u<s,$ and with the adaptedness property : for $t\geq s$,  $\hat{\delta}_s(\overline{\mathcal{F}}_{s,alg}\cap \mathcal{F}_{T-t}))\subset \mathcal{H}( W^*(\overline{\mathcal{F}}_{s,alg}\cap \mathcal{F}_{T-t}),\eta E_B).$ Moreover, $\hat{\Delta}_s=\hat{\delta}_s^*\overline{\hat{\delta}_s}$ extends $\Delta_s:=\delta_s^*\overline{\delta_s}$.
Finally, for any $0\leq s\leq t\leq T$,  $D(\hat{\delta}_s)\subset D(\hat{\delta}_t)$ and  there is some $C\geq 0$ independent of $s,t$ such that, for any $U\in D(\hat{\delta}_s):$ $$||\hat{\delta}_t(U)||\leq C||\hat{\delta}_s(U)||.$$
\item For any $P\in A$, for any $s\leq T$, there exists  paths $(K_t^s(P),L_t^s(P))_{t\in [0,s]}\in C^1([0,s], C^2(A_1,...,A_n:B))^2 $ such that $K_s^s(P)=L_s^s(P)=P$ and for $t\leq s$ we have, with $\Delta_{Q,t}=\Delta_{Q}^{\tau_t}$, the equation in $\mathcal{A}$ :
$$\frac{\partial K_t^s(P)}{\partial t}+\Delta_{Q,t}(K_t^s(P))=0,$$
$$\frac{\partial L_t^s(P)}{\partial t}-\Delta_{Q,T-t}(L_t^s(P))=0.$$

\end{enumerate} 
\end{minipage}

\begin{remark}
\label{continuityFiltration}
Note that the filtration $\mathcal{F}_s=W^*(X_{t},S_t,t\leq s,X\in A)$ is obviously left continuous since $W^*(\mathcal{F}_u,u<s)=\mathcal{F}_s$ since from the definition of Brownian motion and of the solution $X_{t},S_t$ are continuous in $M$. Assumption (C.4) is mainly an assumption of right continuity. Likewise, note that $\overline{\mathcal{F}}_s=W^*(\overline{X}_{t},S_T-S_t,t\leq s,X\in A)$ is also left continuous for the same reason.

\end{remark}

We start by proving an immediate consequence :

\begin{lemma}\label{C7}Assumption (C) implies that for any $P\in A$ and $s\leq t$ :
$$||\delta_{T-s}(\alpha_{s}(K_{s}^{t}(P))||\leq C ||\delta_{T-t}\alpha_{t}(P)||.$$
\end{lemma}
\begin{proof}Using (C.6) and proposition \ref{TimeIto} $$
\alpha_{t}(P)=\alpha_{s}(K_{s}^{t}(P))+\int_s^t\delta_{T-u}(\alpha_{u}(K_{u}^{t}(P))\#dS_u.$$
Since $K_{u}^{t}(P)\in C^2(A_1,...,A_n:B)$ it is easy to see $\delta_{T-u}(\alpha_{u}(K_{u}^{t}(P))\#(S_{u+v}-S_u)\in D(\hat{\delta}_{T-u})$ for $v\geq 0$ and $\hat{\delta}_{T-u}[\delta_{T-u}(\alpha_{u}(K_{u}^{t}(P))\#(S_{u+v}-S_u))] $ is orthogonal to $ \mathcal{H}(\mathcal{F}_u,\eta E_B).$ From (C.5), 
$\delta_{T-u}(\alpha_{u}(K_{u}^{t}(P))\#(S_{u+v}-S_u)\in D(\hat{\delta}_{T-s})$, $u\geq s$. From adaptedness, $\hat{\delta}_{T-s}(\overline{\mathcal{F}}_{T-s,alg}\cap \mathcal{F}_{u}))\subset \mathcal{H}( W^*(\overline{\mathcal{F}}_{T-s,alg}\cap \mathcal{F}_{u}),\eta E_B),$ one sees from the polynomial case that we also have $\hat{\delta}_{T-s}[\delta_{T-u}(\alpha_{u}(K_{u}^{t}(P))\#(S_{u+v}-S_u))] $ is orthogonal to $ \mathcal{H}(\mathcal{F}_s,\eta E_B).$ From the uniformity of the constant $C$ in $(C.5)$ and closability, one extend this to stochastic integrals and see $\int_s^t\delta_{T-u}(\alpha_{u}(K_{u}^{t}(P))\#dS_u\in D(\hat{\delta}_{T-s})$ with $\hat{\delta}_{T-s}(\int_s^t\delta_{T-u}(\alpha_{u}(K_{u}^{t}(P))\#dS_u)$ orthogonal to $ \mathcal{H}(\mathcal{F}_s,\eta E_B).$ As a consequence, $\alpha_{t}(P)\in D(\hat{\delta}_{T-s})$ and from the equation above 

\noindent $E_{\mathcal{H}(\mathcal{F}_s,\eta E_B)}(\hat{\delta}_{T-s}(\alpha_{t}(P)))=\delta_{T-s}(\alpha_{s}(K_{s}^{t}(P))$ and thus applying (C.5) again, one gets the concluding inequalities :
$$||\delta_{T-s}(\alpha_{s}(K_{s}^{t}(P))||\leq ||\hat{\delta}_{T-s}(\alpha_{t}(P)))||\leq C||{\delta}_{T-t}(\alpha_{t}(P)))||.$$
\end{proof}

Let us define our tentative reversed free Brownian motion $$\overline{S}_{i,t}:=S_{i,T-t}-S_{i,T}+\int_0^tds\overline{\xi}_{i,s}.$$
The integral makes sense in Bochner's sense by our assumption.
and $\overline{S}_{i,t}\in M,t<T$. 
 Note that, from our assumptions, one deduces for $u<t$: \begin{align*}||\int_u^tds\overline{\xi}_{i,s}||_2^2&=\int_u^tds_1\int_u^tds_2\tau(\overline{\xi}_{i,s_1}\overline{\xi}_{i,s_2})=2\int_u^tds(t-s)||\overline{\xi}_{i,s}||_2^2+\frac{DC(t-s)^{\alpha+1}}{(\alpha+1)\sqrt{T-t}}\\&\leq2C^2\int_u^tds\frac{t-s}{T-s}+\frac{DC(t-u)^{\alpha+3/2}}{(\alpha+1)(\alpha+3/2)\sqrt{T-t}}=O(t-u).
 \end{align*}

 We now prove :
 
 \begin{proposition}
 With the notation above and under assumption (C.0,1,2,3,6), $\overline{S}_{i,t}t\in[0,T]$ is a free Brownian motion relative to $B$ of covariance $\eta$ adapted to the filtration $\overline{\mathcal{F}}_s=W^*(B,\alpha_{T-t}(X), X\in A,t\in [0,s],\overline{S}_{i,t},t\in [0,s])$.
 \end{proposition}
\begin{proof}We will also write $\mathcal{F}_s=W^*(B,\alpha_{t}(X), X\in A,t\in [0,s],{S}_{i,t},t\in [0,s])$
We use Paul L\'evy's Thm for free Brownian motion \cite{Dab10,BGC} recalled in Theorem \ref{FreeLevy} above. We first check this on $[0,\tau]$ for $\tau<T$. Fix $t>s$.

Of course $\overline{S}_{i,0}=0$ and since $\overline{S}_{i,t}-\overline{S}_{i,s}={S}_{i,T-t}-{S}_{i,T-s}+\int_s^tdu\overline{\xi}_{i,u}$. Using free markovianity in Assumption (B.2),   $\mathcal{F}_{T-s}$ and $\overline{\mathcal{F}}_s$ are free with amalgamation over $\mathcal{F}_{T-s}\cap\overline{\mathcal{F}}_s=W^*(B,\alpha_{T-s}(X), X\in A)$ so that, since $\overline{S}_{i,t}-\overline{S}_{i,s}\in \mathcal{F}_{T-s}$, we can compute  $E(\overline{S}_{i,t}-\overline{S}_{i,s}|\overline{\mathcal{F}}_s)=E({S}_{i,T-t}-{S}_{i,T-s}+\int_s^tdu\overline{\xi}_{i,u}|W^*(\alpha_{T-s}(A))$. 

Using Proposition \ref{TimeIto} and assumption (C.6), one gets for any $X\in A$:

\begin{align}\label{ItoPDE}\alpha_{T-s}(X)=\alpha_{T-s}(K_{T-s}^{T-s}(X))=\alpha_{0}(K_{0}^{T-s}(X))+\int_0^{T-s}(\alpha_{u}\o \alpha_u)(\delta(K_{u}^{T-s}(X)))\#dS_u.\end{align}
Write $X_v^{T-s}=\alpha_{0}(K_{0}^{T-s}(X))+\int_0^{v}(\alpha_{u}\o \alpha_u)(\delta(K_{u}^{T-s}(X)))\#dS_u,$ for the version stopped at $v\leq T-s.$ Comparing the equation above with its analogue on $[v,T-s]$, we easily get the alternative formula : $X_v^{T-s}=\alpha_v(K_{v}^{T-s}(X)).$

By Ito formula again, if we consider the process $Y_{T-s}={S}_{i,T-t}-{S}_{i,T-s}+\int_s^tdu\overline{\xi}_{i,u}$, for $T-s\geq T-t$, one can compute for any $X\in A$:
\begin{align*}Y_{T-s}\alpha_{T-s}(X)&=-\int_{T-t}^{T-s}dS_{i,u}X_u^{T-s}-Y_{T-u}\alpha_u\otimes\alpha_u(\delta(K_{u}^{T-s}(X)))\#dS_{u}\\&+\int_{T-t}^{T-s}du[\overline{\xi}_{i,u}X_u^{T-s}-m(\eta E_B\alpha_u)\otimes\alpha_u\delta(K_{u}^{T-s}(X))]\end{align*}

Taking the trace and using the definition of $\overline{\xi}_{i,u}$ one deduces :
$$\tau(Y_{T-s}\alpha_{T-s}(X))=0
.$$

Thus, we have checked using the PDE in (C.6) the crucial martingale property, $E(\overline{S}_{i,t}-\overline{S}_{i,s}|\overline{\mathcal{F}}_s)=0$ as expected.


Then we have to estimate the $L^4$ norm :

$||\overline{S}_{i,t}-\overline{S}_{i,s}||_4\leq 2\sqrt{t-s}+||\int_s^tdu\overline{\xi}_{i,u}||_4.$

$$||\int_s^tdu\overline{\xi}_{i,u}||_4^4=\int_s^tdu_4\int_s^{u_4}du_1\int_s^{u_1}du_2\int_s^{u_2}du_34\sum_{\sigma\in\mathscr{S}_3}\tau(\overline{\xi}_{i,u_{\sigma(1)}}\overline{\xi}_{i,u_{\sigma(2)}}\overline{\xi}_{i,u_{\sigma(3)}}\overline{\xi}_{i,u_1}).
$$ 
so that , for $s\leq t\leq \tau<T$, using (C.3) to project the largest time $u_4$ to the second largest $u_1$ using free markovianity, and then use (C.2) on other times, we get :
\begin{align*}||&\int_s^tdu\overline{\xi}_{i,u}||_4^4\\&\leq \int_s^tdu_1\int_s^{u_1}du_2\frac{C}{\sqrt{T-u_2}}\int_s^{u_2}du_3\frac{C}{\sqrt{T-u_3}}\times24\left((t-u_1)||\overline{\xi}_{i,u_1}||_2^2+\frac{CD(t-u_1)^{\alpha+1}}{(\alpha+1)\sqrt{T-t}}\right)\\&\leq 12C^2\int_s^tdu_1(u_1-s)^2\frac{(t-u_1)}{T-u_1}||\overline{\xi}_{i,u_1}||_2^2+12C^2(t-s)^2\frac{CD(t-s)^{\alpha+1}}{(\alpha+1)^2\sqrt{T-t}}
\\&\leq 12C^2(t-s)^2\int_0^\tau du_1||\overline{\xi}_{i,u_1}||_2^2+24C^2(t-s)^2\frac{CD(t-s)^{\alpha+1}}{(\alpha+1)^2\sqrt{T-\tau}}.
\end{align*} 
The last term is a $O((t-s)^2)$ for $\tau<T$ as expected

For later convenience, let us write $E_s=E_{W^*(B,\alpha_{T-s}(A))}$.

It remains to compute for $D,E\in \overline{\mathcal{F}}_s$ a covariance term but from free markovianity, one gets for $s<t$:
\begin{align*}&\tau( (\overline{S}_{i,t}-\overline{S}_{i,s})D(\overline{S}_{i,t}-\overline{S}_{i,s})E)
\\&=\tau(({S}_{i,T-t}-{S}_{i,T-s}+\int_s^tdu\overline{\xi}_{i,u})E_{s}(D)({S}_{i,T-t}-{S}_{i,T-s}+\int_s^tdu\overline{\xi}_{i,u})E_{s}(E))
\\& =\tau(({S}_{i,T-s}-{S}_{i,T-t})[E_{t}(E_{s}(D))]({S}_{i,T-s}-{S}_{i,T-t})E_{t}(E_{s}(E))])\\ &+\tau(({S}_{i,T-t}-{S}_{i,T-s})[E_{s}(D)-E_t(E_s(D))]({S}_{i,T-t}-{S}_{i,T-s})E_{t}(E_{s}(E))])\\ &+
\tau(({S}_{i,T-t}-{S}_{i,T-s})[E_{t}(E_{s}(D))]({S}_{i,T-t}-{S}_{i,T-s})[E_{s}(E)-E_{t}(E_{s}(E))])\\ &+\tau(({S}_{i,T-t}-{S}_{i,T-s})[E_{s}(D)-E_{t}(E_{s}(D))]({S}_{i,T-t}-{S}_{i,T-s})[E_{s}(E)-E_{t}(E_{s}(E))])
\\&+\tau(({S}_{i,T-t}-{S}_{i,T-s})E_{s}(D)(\int_s^tdu\overline{\xi}_{i,u})E_{s}(E))
\\&+\tau((\int_s^tdu\overline{\xi}_{i,u})E_{s}(D)({S}_{i,T-t}-{S}_{i,T-s})E_{s}(E))
\\&+\tau((\int_s^tdu\overline{\xi}_{i,u})E_{s}(D)(\int_s^tdu\overline{\xi}_{i,u})E_{s}(E))
\\&=\tau(\eta (E_B([E_{t}(E_{s}(D))]))[E_{t}(E_{s}(E))])+o(t-s)\\&=\tau(\eta (E_B(D))E)+o(t-s).\end{align*}
We used in the next-to-last line that ${S}_{i,T-s}$ is a free Brownian motion via its covariance and the estimates  $||{S}_{i,T-s}- {S}_{i,T-t}||\leq 2\sqrt{t-s}$, $(t-s)||(E_s-E_t)(C)||_2=(t-s)||C-E_{\mathcal{F}_t}(C)||_2=o(t-s)$ (for $C\in W^*(B,\alpha_{T-s}(A))$ following by Kaplansky density Theorem from the case $C\in B\langle\overline{X}^i_s,i=1,...,N\rangle$) and :
\begin{align*}|\int_s^tdu_1\int_s^{t}du_2&\tau((\overline{\xi}_{i,u_1})E_{s}(A)(\overline{\xi}_{i,u_2})E_{s}(D))|\\&=2\int_s^tdu|(t-u)\tau(\overline{\xi}_{i,u}E_{s}(A)\overline{\xi}_{i,u}E_{s}(D))|+\frac{CD(t-u)^{\alpha+1}}{(\alpha+1)\sqrt{T-t}}||A||\ ||D||\\&\leq 2||A||\ ||D||[(t-s)\int_s^tdu||\overline{\xi}_{i,u}||_2^2+\frac{CD(t-s)^{\alpha+2}}{(\alpha+1)(\alpha+2)\sqrt{T-t}}]=o(t-s),\end{align*}
with again the last estimate for $s<t\leq \tau<T$ and similarly (case $A=D=1$) $||\int_s^tdu\overline{\xi}_{i,u}||_2^2=o(t-s)$. This concludes since having a free Brownian motion on $[0,\tau]$ for any $\tau<T$ easily gives a free Brownian motion on $[0,T]$, using the improved estimates obtained on $[0,T)$ (from the newly proven fact that we have a free Brownian motion), their extension by continuity to $[0,T]$ and the free Paul L\'evy's Thm again). 
\end{proof}
 \subsection{Preliminaries on certain unbounded operators and stochastic integrals}
 We now want to obtain a SDE for the reversed process.

To improve our previous result and deduce regularity results for the conjugate variable, we will need several technical preliminaries gathered in the following lemma.

We will also need to consider an ad hoc space of ``Regular bi-processes'' adapted to the reversed filtration, in order to manipulate first Riemann like sums before applying density results.

\begin{align*}
&\mathcal{R}([0,T])=\{s\mapsto V_s\in (D(\hat{\delta_s})\cap\overline{\mathcal{F}}_s)\o_{alg}(D(\hat{\delta_s})\cap\overline{\mathcal{F}}_s)\ |\ s\mapsto V_s\in C^{1}([0,T],M\hat{\o}M),\\& s\mapsto (\hat{\delta_s}\otimes 1(V_s))\oplus (1\otimes \hat{\delta_s})(V_s)\in \mathcal{B}([0,T],\mathcal{H}(M,\eta\circ E_B)\hat{\o}M\oplus M\hat{\o}{H}(M,\eta\circ E_B)),\\ & s\mapsto (\hat{\delta_s}\otimes \hat{\delta_s}(V_s))\in \mathcal{B}(([0,T],\mathcal{H}(M,\eta\circ E_B)\hat{\o}\mathcal{H}(M,\eta\circ E_B)),\\&
(s,u)\mapsto (\hat{\delta_s}\otimes 1(V_u))\oplus (1\otimes \hat{\delta_s})(V_u)\in
\\ & \mathcal{B}([0,T],C^1([0,.],\mathcal{H}(M,\eta\circ E_B)\hat{\o}L^2(M)\oplus L^2(M)\hat{\o}{H}(M,\eta\circ E_B))
\},
\end{align*}
where $\mathcal{B}$ denotes a space of bounded function, for instance the last space has a norm :

$$\sup_{s\in [0,T]}\sup_{u\in[0,s]}||(\hat{\delta_s}\otimes 1(V_u))\oplus (1\otimes \hat{\delta_s})(V_u)||+||\frac{\partial}{\partial u}(\hat{\delta_s}\otimes 1(V_u))\oplus (1\otimes \hat{\delta_s})(V_u)||.$$

Also recall that $L^\infty([0,T],M):=L^\infty([0,T],dLeb)\bar{\o} M$ the von Neumann algebraic tensor product has predual $L^1([0,T],M)$ defined in Bochner's sense, but that $L^\infty([0,T],L^1(M))$ is defined (when $M$ has separable predual) in a measure theoretic way as the space of weak-* scalarly measurable functions bounded in the right norm (which is not $L^\infty$ in Bochner's sense). Recall also we wrote $E_s=E_{W^*(B,\alpha_{T-s}(A))}$. We will use freely non-commutative Dirichlet forms techniques, see e.g. \cite{CiS}. Note that we will use crucially that the non-reversed filtration is better behaved e.g. continuous. (C.5) can also be seen as such a regularity in the original direction of time.

We gather our preliminary technicalities in the next :

\begin{lemma}\label{Unbounded}Assume assumption (C).
\begin{enumerate}
\item Let $s\in [0,T]$, then for any $U\in \overline{\mathcal{F}}_s\cap D(\hat\delta_s)$,  we have $E_s(U)\in D(\delta_s)$ and $||\delta_{s}E_{s}(U)||\leq ||\hat{\delta}_{s}(U)||.$ 
\item $L^2_{biad}([0,T],L^2(M))=\{f\in L^2([0,T],L^2(M)) f(s)\in L^2(W^*(\alpha_{T-s}(A))) a.e.\}$ is a closed subspace of $L^2([0,T],L^2(M))$ (If $P_s$ is the projection on $L^2(W^*(\alpha_{T-s}(A)))$, $P:f\mapsto(s\mapsto P_s(f(s))$ is the corresponding projection).
Likewise $L^2_{biad}([0,T],\mathcal{H}(M,\eta\circ E_B))=\{f\in L^2([0,T],\mathcal{H}(M,\eta\circ E_B)) f(s)\in \mathcal{H}(W^*(\alpha_{T-s}(A),\eta\circ E_B) a.e.\}$ is a closed subspace of $L^2([0,T],\mathcal{H}(M,\eta\circ E_B)).$
\item Let us define $$D(\mathcal{E})=\{f\in L^2_{biad}([0,T],L^2(M)), f(s)\in D(\overline{\delta_s})a.e.,s\mapsto\overline{\delta_s}f(s)\ \text{measurable}\ \int_0^T || \overline{\delta_s}f(s)||_2^2ds<\infty\}.$$ Then $\mathcal{E}(f)=\int_0^T || \overline{\delta_s}f(s)||_2^2ds$ is a closed form on $D(\mathcal{E}),$ defining an operator $\delta_{[0,T]}:D(\mathcal{E})\to 
L^2_{biad}([0,T],\mathcal{H}(M,\eta\circ E_B))$ by $\delta(f))(s)=\delta_s(f(s)).$
\item 
 The generator $\Delta$ of $\mathcal{E}$ is characterized as follows (we write $\Delta_s={\partial_s}^*\overline{\partial_s}$)
 :

$$D(\Delta)=\{f\in L^2_{biad}([0,T],L^2(M)), f(s)\in D({\Delta_s})a.e.,s\mapsto\Delta_sf(s)\ \text{meas.}\ \int_0^T || \Delta_sf(s)||_2^2ds<\infty\}$$

and  
 for any $f\in D(\Delta)$, $\Delta(f)(s)=\Delta_s(f(s))$ a.e.
Likewise,  $$D(\delta_{[0,T]}^*)=\{f\in L^2_{biad}([0,T],\mathcal{H}(M,\eta\circ E_B)), f(s)\in D(\overline{\delta_s}^*)a.e.,s\mapsto\delta_s^*f(s)\ \text{meas.}\ \int_0^T || \delta_s^*f(s)||_2^2ds<\infty\}$$
 and for any $f\in D(\delta_{[0,T]}^*)$ for a.e. $s$ $(\delta_{[0,T]}^*f)(s)=\delta_s^*(f(s))$. As a consequence, for any $f\in D(\Delta)$, $\delta_{[0,T]}f\in D(\delta_{[0,T]}^*)$ and $\Delta f=\delta_{[0,T]}^*\delta_{[0,T]}f.$
\item 
Let $U_s=E_s(V_s))$ with $V\in \mathcal{R}([0,T])$. 
Then  for $v<T$ we have :$$\int_u^v U_s\#d\overline{S}_s +\int_{T-v}^{T-u}U_{T-s}\#d{S}_s-\int_u^v\delta^*_s(U_s)ds=0.$$
 \item $\{s\mapsto U_s=E_s(V_s)):V\in \mathcal{R}([0,T])\}$ is dense in $L^2_{biad}([0,T],\mathcal{H}(M,\eta\circ E_B))$.
\item  $L^\infty_{biad}([0,T],M):=L^2_{biad}([0,T],L^2(M))\cap 
 L^\infty([0,T],M)$ is a finite von Neumann algebra with GNS construction $L^2_{biad}([0,T],L^2(M))$ with respect to the canonical induced trace $\tau_{[0,T]}(f)=\int_0^T\tau(f(s))ds,$ and with predual $$L^1_{biad}([0,T],L^1(M)):=\overline{L^2_{biad}([0,T],L^2(M))}^{L^1}\subset L^1([0,T],L^1(M)).$$ $\mathcal{E}$ is a (completely) Dirichlet form on $L^\infty_{biad}([0,T],M)$ and the canonical $L^1$-extension of its generator (cf. \cite[section 1.4]{Pe04}, called $\Psi$ there)  has domain :
$$D(\Delta^1)=\{f\in L^\infty_{biad}([0,T],M), f(s)\in D({\Delta_s^1})a.e.,s\mapsto\Delta_s^1f(s)\ \text{meas.}\ \int_0^T || \Delta_s^1f(s)||_1ds<\infty\},$$

and moreover, for any $f\in D(\Delta^1)$ (the corresponding $L^1$-extension for $\Delta$), for almost every $s$, $\Delta^1(f)(s)=\Delta^1_sf(s).$
\end{enumerate}
\end{lemma} 

 \begin{proof}
 \begin{enumerate}
 \item Take $V\in D(\Delta_s)$, so that by assumption (C.5) $$\langle(\hat{\delta}_s(U)),\delta_s(V)\rangle=\langle U,\hat{\Delta}_s(V)\rangle=\langle U,\Delta_s(V)\rangle=\langle E_u(U),\Delta_s(V)\rangle=\langle {\delta}_s(E_s(U)),\delta_s(V)\rangle,$$
 where the next to last equality shows $E_s(U)\in D(\delta_s)$, and the inequality by duality.
 \item Take $f_n\in L^2_{biad}([0,T],L^2(M))$ converging to $f$ in $L^2$. Up to extraction it also converges almost everywhere but a.e. $P_sf_n(s)=f_n(s)$ ($\forall n$) and $||P_sf_n-P_sf||_2^2(s)\to 0$, i.e. by uniqueness of the limit $P_sf(s)=f(s)$ almost everywhere.

 For the statement about $P$, first define $P$ on the dense subset $L^\infty([0,T],\C)\otimes L^2(M,\tau)$ with value $L^2_{biad}([0,T],L^2(M))$. For this we have to check that for a constant $b\in L^2(M)$ $P_s(b)$ is Bochner measurable, then it will be obviously in  $L^2_{biad}([0,T],L^2(M))$ and the same will be true for the value of the map on the dense set above. But $P_s=E_{\overline{\mathcal{F}}_s}E_{\mathcal{F}_s}$ (by freeness with amalgamation in assumption (B.2)) and by assumption $\mathcal{F}_s$ is a continuous filtration i.e. $E_{\mathcal{F}_s}(b)$ is continuous in $L^2(M)$ thus again approximated by density by the set $L^\infty([0,T],\C)\otimes L^2(M,\tau)$ so that it suffices to check that for $b\in L^2(M,\tau)$, $E_{\overline{\mathcal{F}}_s}(b)$ is also Bochner measurable. But as we noted in remark \ref{continuityFiltration}, $\overline{\mathcal{F}}_s$ is left continuous thus so is $E_{\overline{\mathcal{F}}_s}(b)$ in $L^2$, thus it is Bochner measurable, as we wanted. We can now easily extend $P$ to the operator stated in the statement.

 \item Take a Cauchy sequence $f_n$ in $D(\mathcal{E})$ with norm $\mathcal{E}_1^{1/2}$. Recall $\mathcal{E}_1(x)=||x||_2^2+ \mathcal{E}(x)$ is the squared graph norm or a coercive version of $\mathcal{E}$. Thus $f_n,  s\mapsto\overline{\partial_s}f_n(s)$ converge to $f, g$ in $L^2$, and moreover modulo extraction, they converge almost surely and in $L^2$. Since $\overline{\partial_s}$ is closed, on the almost sure convergence set $f(s)\in D(\overline{\partial_s})$, $g(s)=\overline{\partial_s}f(s)$ this concludes to $f\in D(\mathcal{E})$ and $D(\mathcal{E})$ complete.
 \item
The fact that $D(\Delta)$ is included in the domain of the generator is obvious. Conversely, take $f$ in the domain of the generator  so that by definition $f\in D(\mathcal{E})$ and for any $v\in D(\mathcal{E})$ $v\mapsto \int_0^Tds\langle \delta_sf(s),\delta_sv(s)\rangle$ defines a continuous linear form on a dense subset of $L^2_{biad}$ thus on $L^2_{biad}$. Thus by duality, we have $w\in L^2_{biad}$ such that  $\int_0^Tds\langle \delta_s(f(s)),\delta_sv(s)\rangle=\int_0^Tds\langle w(s),v(s)\rangle.$ Since for any $g\in L^\infty([0,T],\C)$, $gv$ is again in $D(\mathcal{E})$, we get by scalar-$L^p$ spaces duality that 
$\langle \delta_s(f(s)),\delta_sv(s)\rangle=\langle w(s),v(s)\rangle$ as function of $s$ in $L^{1}([0,T],\C)$
thus especially almost everywhere.

Now we can take the specific $v(s)=\alpha_{T-s}(P)$ for $P$ in a dense countable $\Q$-subalgebra of $A$ which is still a core of $\delta_s$ after taking the image by $\alpha_{T-s}$ for any $s$. Thus, in the equality above $v(s)$ can be replaced a.e. by any $v\in D(\delta_s)$ and since a.e. $||w(s)||_2<\infty$  one deduces $f(s)\in D(\Delta_s)$ a.e. and $\Delta_s(f(s))=w_s$. Thus $s\mapsto\Delta_s(f(s)) \in  L^2_{biad}$. This concludes the computation of the domain and the equality $(\Delta f)(s):=w(s)=\Delta_s(f(s)).$

The proof of the second statement is similar.

 \item 
 We have by construction and continuity (right and left) of the filtration $\mathcal{F}_s$ the continuity of $U_s=E_{\mathcal{F}_{T-s}}(V_s)$ say in $L^2$ so that :
  $$\int_v^u U_s\#d\overline{S}_s=\lim_{m\to\infty}\sum_{k=1}^m(U_{v+k(u-v)/m})\#(\overline{S}_{v+(k+1)(u-v)/m}-\overline{S}_{v+k(u-v)/m})$$
  
   $$\int_{T-u}^{T-v}U_{T-s}\#d{S}_s=\lim_{n\to\infty}\sum_{k=1}^m(U_{v+(k+1)(u-v)/m}\#({S}_{T-(v+k(u-v)/m)}-{S}_{T-(v+(k+1)(u-v)/m)})$$
We have to compute the quadratic variation like quantity :

\begin{align*}   &\int_v^u U_s\#d\overline{S}_s+\int_{T-u}^{T-v}U_{T-s}\#d{S}_s=\lim_{m\to\infty}\\&-\sum_{k=1}^m(U_{v+k(u-v)/m}-U_{v+(k+1)(u-v)/m}))\#({S}_{T-(v+k(u-v)/m)}-{S}_{T-(v+(k+1)(u-v)/m)})
\\&+\sum_{k=1}^m(U_{v+k(u-v)/m}))\#(\int_{v+k(u-v)/m}^{v+(k+1)(u-v)/m}\overline{\xi}_{w}dw)
\end{align*}

where we wrote with a small abuse of notation : $U_{v+k(u-v)/m}-U_{v+(k+1)(u-v)/m}=R_{k,m}+W_{k,m}$ with $$R_{k,m}=\iota(E_{v+(k+1)(u-v)/m}(V_{v+k(u-v)/m}-V_{v+(k+1)(u-v)/m})),$$ $$W_{k,m}=\iota(
E_{v+k(u-v)/m}(V_{v+k(u-v)/m})-E_{v+(k+1)(u-v)/m}(V_{v+k(u-v)/m})),$$ with $E_{v+k(u-v)/m}$ the tensor product of the corresponding conditional expectation in projective tensor product, $\iota$ the canonical map for $M\hat{\o}M $ to $L^2(M\o M)$. Since $V$ is $C^1$, $||R_{k,m}||\leq \frac{1}{m}\sup_{t\in [0,T]}||\frac{\partial}{\partial t}V_{t}||$ so that $\lim_{m\to\infty}\sum_{k=1}^mR_{k,m}\#({S}_{T-(v+k(u-v)/m)}-{S}_{T-(v+(k+1)(u-v)/m)})=0$ e.g. in $M.$ Let us write for convenience :$S^{(k,n)}:=({S}_{T-(v+k(u-v)/m)}-{S}_{T-(v+(k+1)(u-v)/m)}$

We now estimate $$[A]:=||\sum_{k=1}^mW_{k,m}\#S^{(k,n)}-\sum_i\int_u^vdw(m\circ(1\otimes \eta\circ E_B\otimes 1)(\overline{\delta}_{i,w}\otimes 1+1\o\overline{\delta}_{i,w})(U_{w})\|_1$$
We first consider the case in which $V_{v+k(u-v)/m}=a_{k,m}\o b_{k,m}$ and for convenience we write $E_{k,m}=E_{v+k(u-v)/m}$.

We know from equation \eqref{ItoPDE} extended from $X$ to $E_{k,m}(a_{k,m})$ that $$(E_{k,m}-E_{k+1,m})a_{k,m}=\int^{T-(v+k(u-v)/m)}_{T-(v+(k+1)(u-v)/m)}(\delta_{T-s}(E_{T-s}\circ E_{k,m}(a_{k,m}))\#dS_s.$$ By point (1) we have just proved $||\delta_{(v+k(u-v)/m)}E_{k,m}(a_{k,m})||\leq ||\hat{\delta}_{(v+k(u-v)/m)}(a_{k,m})||$ and  moreover using equation (\ref{ItoPDE}) again and by lemma \ref{C7} for $T-s\leq T-(v+k(u-v)/m):$  $$||\delta_{s}E_{s}\alpha_{T-(v+k(u-v)/m)}(P)||=||\delta_{s}(\alpha_{T-s}K_{T-s}^{T-v+k(u-v)/m)}(P)||\leq C ||\delta_{(v+k(u-v)/m)}\alpha_{T-(v+k(u-v)/m)}(P)||_2$$  for $P\in A$ so that the inequality extends to $D(\delta_{(v+k(u-v)/m)})$ and thus :

$$||(E_{k,m}-E_{k+1,m})a_{k,m}||_2\leq C\sqrt{\frac{1}{m}}||\hat{\delta}_{(v+k(u-v)/m)}(a_{k,m})||,$$

\begin{align*}||&\sum_{k=1}^m(E_{k,m}-E_{k+1,m})(a_{k,m})({S}^{(k,n)})(E_{k,m}-E_{k+1,m})(b_{k,m})||_1\\&\leq 2C\frac{1}{\sqrt{m}}\sup_k||\hat{\delta}_{(v+k(u-v)/m)}(a_{k,m})||\ ||\hat{\delta}_{(v+k(u-v)/m)}(b_{k,m})||\end{align*}

From Ito formula (version of Proposition \ref{ItoSimple}) applied to $S$ and \eqref{ItoPDE} and from orthogonality, we know that :

\begin{align*}&[B]=||\sum_{k=1}^m[(E_{k+1,m}(a_{k,m}) ({S}^{(k,n)}(E_{k,m}-E_{k+1,m})(b_{k,m}))\\ &-\sum_k\int_{T-(v+(k+1)(u-v)/m)}^{T-(v+k(u-v)/m)}dw(m\circ(1\otimes \eta\circ E_B\otimes 1)(E_{T-w}(a_{k,m})\otimes(\delta_{T-w}(E_{T-w}\circ E_{k,m}(b_{k,m}))] \|_1\\&\leq\left(\sum_{k=1}^m||\int_{T-(v+(k+1)(u-v)/m)}^{T-(v+k(u-v)/m)}(E_{k+1,m}(a_{k,m})dS_w(E_{T-w}-E_{k+1,m})(b_{k,m}))\right.\\ &\left.+\int_{T-(v+(k+1)(u-v)/m)}^{T-(v+k(u-v)/m)}(E_{k+1,m}(a_{k,m})(S_w-{S}_{T-(v+(k+1)(u-v)/m)})\delta_{T-w}(E_{T-w})(b_{k,m}))\#dS_w\|_2^2\right)^{1/2}\\ &+||\sum_k\int_{T-(v+(k+1)(u-v)/m)}^{T-(v+k(u-v)/m)}dw(m\circ(1\otimes \eta\circ E_B\otimes 1)((E_{T-w}-E_{k+1,m})(a_{k,m})\otimes(\delta_{T-w}(E_{T-w}(b_{k,m}))] \|_1
\end{align*}
so that \begin{align*}[B]&\leq \sqrt{2\sum_{k=1}^m||a_{k,m}||_2^2\int_{T-(v+(k+1)(u-v)/m)}^{T-(v+k(u-v)/m)}dw||(E_{T-w}-E_{k+1,m})(b_{k,m}))||_2^2}\\ &+\sqrt{2\sum_{k=1}^m||a_{k,m}||^2\int_{T-(v+(k+1)(u-v)/m)}^{T-(v+k(u-v)/m)}dw||S_w-{S}_{T-(v+(k+1)(u-v)/m)})||^2 ||\delta_{T-w}(E_{T-w})(b_{k,m})\|_2^2}
\\&+\sum_{k=1}^m\int_{T-(v+(k+1)(u-v)/m)}^{T-(v+k(u-v)/m)}dw||((E_{T-w}-E_{k+1,m})(a_{k,m})||_2||\delta_{T-w}(E_{T-w}(b_{k,m}))] \|_2
\\&
\leq \frac{4C}{m}\sqrt{\sum_{k=1}^m||a_{k,m}||^2||\hat{\delta}_{(v+k(u-v)/m)}(b_{k,m})||^2}
\\&+ \frac{C}{m\sqrt{m}}\sum_{k=1}^m||\hat{\delta}_{(v+k(u-v)/m)}(a_{k,m})||||\hat{\delta}_{(v+k(u-v)/m)}(b_{k,m})||\end{align*}

where we used in the last line our previously shown inequalities.

We will  of course also  use a symmetric bound. 

All the previous inequalities can be extended to any $V$ by linearity and density.
It remains to bound :
\begin{align*}||\sum_{k=1}^m\sum_i&\int_{T-(v+(k+1)(u-v)/m)}^{T-(v+k(u-v)/m)}dw(m\circ(1\otimes \eta\circ E_B\otimes 1)1\otimes\overline{\delta}_{i,T-w}[(E_{T-w}\o E_{T-w})(V_{v+k(u-v)/m}-V_{T-w})]\|_1\\&\leq \sqrt{n}\sum_{k=1}^m\int_{T-(v+(k+1)(u-v)/m)}^{T-(v+k(u-v)/m)}dw||(1\otimes\hat{\delta}_{T-w})(V_{v+k(u-v)/m}-V_{T-w})]\|_{L^2(M)\hat{\o}\mathcal{H}(M,\eta\circ E_B)}
\\&\leq \frac{\sqrt{n}}{m}\sup_{w\in[0,T]}\sup_{s\in[0,w]}||\frac{\partial}{\partial s}(1\otimes\hat{\delta}_w)(V_{s})]\|_{L^2(M)\hat{\o}\mathcal{H}(M,\eta\circ E_B)}\end{align*}

Summarizing our estimates, we got :

\begin{align*}[A]&\leq 2(1+n)C\frac{1}{\sqrt{m}}\sup_{s\in[u,v]}||\hat{\delta}_{s}\o\hat{\delta}_{s}(V_s)||_{\mathcal{H}(M,\eta\circ E_B)\hat{\o}\mathcal{H}(M,\eta\circ E_B)}
\\&+\frac{4C}{\sqrt{m}}\sup_{s\in[u,v]}\left(||(1\otimes\hat{\delta}_s)(V_{s})]\|_{M\hat{\o}\mathcal{H}(M,\eta\circ E_B)}+||(\hat{\delta}_s\o 1)(V_{s})]\|_{\mathcal{H}(M,\eta\circ E_B)\hat{\o} M}\right)
\\&+\frac{\sqrt{n}}{m}\sup_{w\in[0,T]}\sup_{s\in[0,w]}||\frac{\partial}{\partial s}(1\otimes\hat{\delta}_w)(V_{s})]\|_{L^2(M)\hat{\o}\mathcal{H}(M,\eta\circ E_B)}+||\frac{\partial}{\partial s}(\hat{\delta}_w\o 1)(V_{s})]\|_{\mathcal{H}(M,\eta\circ E_B)\hat{\o}L^2(M)}
\end{align*}

And the bound tends to zero as expected.

Similarly, one can estimate for $s\in [v+k(u-v)/m,v+(k+1)(u-v)/m]$ :
\begin{align*}||&(U_{i,v+k(u-v)/m}-(U_{i,s})||_{L^2(M)\hat{\o}L^2(M)}\leq \frac{1}{m}\sup_{t\in [0,T]}||\frac{\partial}{\partial t}V_{t}||_{M\hat{\o}M}\\&+\frac{C^2}{m}||\hat{\delta}_{v+k(u-v)/m}\o\hat{\delta}_{v+k(u-v)/m}(V_{v+k(u-v)/m})||_{\mathcal{H}(M,\eta\circ E_B)\hat{\o}\mathcal{H}(M,\eta\circ E_B)}\\&+\frac{C}{\sqrt{m}}\left(||1\o\hat{\delta}_{v+k(u-v)/m}(V_{v+k(u-v)/m})||_{M\hat{\o}\mathcal{H}(M,\eta\circ E_B)}+||\hat{\delta}_{v+k(u-v)/m}\o 1(V_{v+k(u-v)/m})||_{\mathcal{H}(M,\eta\circ E_B)\hat{\o}M}\right)\end{align*}

Finally we bound :

\begin{align*}||\sum_{k=1}^m(U_{v+k(u-v)/m}))&\#(\int_{v+k(u-v)/m}^{v+(k+1)(u-v)/m}\overline{\xi}_{w}dw-\sum_i\int_u^v(U_{i,s})\#\overline{\xi}_{i,s}ds||_1\\&\leq \sum_{k=1}^m\int_{v+k(u-v)/m}^{v+(k+1)(u-v)/m}\sum_i||(U_{i,v+k(u-v)/m}-(U_{i,s})\#\overline{\xi}_{i,s}||_1ds
\\&\leq \frac{Cn}{m\sqrt{T-v}}\sup_{t\in [0,T]}||\frac{\partial}{\partial t}V_{t}||_{M\hat{\o}M}\\&+\frac{C^3n}{m\sqrt{T-v}}\sup_{t\in [0,T]}||\hat{\delta}_{t}\o\hat{\delta}_{t}(V_{t})||_{\mathcal{H}(M,\eta\circ E_B)\hat{\o}\mathcal{H}(M,\eta\circ E_B)}\\&+\frac{C^2n}{\sqrt{(T-v)m}}\sup_{t\in [0,T]}\left(||1\o\hat{\delta}_{t}(V_{t})||_{M\hat{\o}\mathcal{H}(M,\eta\circ E_B)}+||\hat{\delta}_{t}\o 1(V_{t})||_{\mathcal{H}(M,\eta\circ E_B)\hat{\o}M}\right),
\end{align*}
where we used (C.2) to bound $||\overline{\xi}_{i,s}||.$ Putting everything together and using Voiculescu's formula in lemma \ref{VoiculescuFormula}, this concludes.

\item From free markovianity and (1), its suffices to prove $\mathcal{R}([0,T])$ is dense in $L^2_{ad}([0,T],\mathcal{H}(\overline{\mathcal{F}},\eta\circ E_B)$. It thus suffices to approximate an adapted simple processes $$V_s=\sum_{k=0}^{m-1}1_{[kT/m,(k+1)T/m)}(s)V_{(k)}$$ by $\mathcal{R}([0,T])$ and one can assume by density $V_{(k)}\in (D(\hat{\delta}_{kT/m})\cap\overline{\mathcal{F}}_{kT/m})\o_{alg}(D(\hat{\delta}_{kT/m})\cap\overline{\mathcal{F}}_{kT/m}),$
 so that from assumption (C.5) $V_s\in (D(\hat{\delta_s})\cap\overline{\mathcal{F}}_s)\o_{alg}(D(\hat{\delta_s})\cap\overline{\mathcal{F}}_s)$ for any $s$.
Now $V_t$ is approximated by $V*\varphi_l$ the adapted time convolution with the smooth functions $\varphi_l(x)=l\varphi(lx)$ for $\varphi$ smooth non-negative supported on $[0,1]$ of integral $1$ (cf. proof of lemma \ref{simple}).

It thus remains to check  $V*\varphi_n\in \mathcal{R}([0,T])$
 
 Note that again $(V*\varphi_l)_s\in (D(\hat{\delta_s})\cap\overline{\mathcal{F}}_s)\o_{alg}(D(\hat{\delta_s})\cap\overline{\mathcal{F}}_s)$ since for $s\in [0,T]:$ $$(V*\varphi_l)_s=\int_0^{T\wedge s}duV_u\varphi_l(s-u)du=\sum_{k=0}^{m-1(\wedge sm/T)}V_{(k)}(\int_{[kT/m,(k+1)T/m)}du\varphi_l(s-u) ).$$
 From this expression and the last inequality in assumption (C.5),  $(V*\varphi_l)\in \mathcal{R}([0,T])$ is now easy so that we leave the details to the reader.

\item Clearly, $L^\infty_{biad}([0,T],M)$ is a  von Neumann subalgebra $L^\infty([0,T],M).$ Indeed, it is clearly a $*$-subalgebra and if $v_n\in L^\infty_{biad}([0,T],M)$ converges weakly in $L^\infty([0,T],M)$ it especially converges in $L^2([0,T],L^2(M))$ thus by the point (1) above  the limit is still in $L^2_{biad}([0,T],L^2(M))$ and thus in $L^\infty_{biad}([0,T],M).$  The statement about the finite subalgebra property and the predual follows once the GNS construction identified. But for any positive $f\in L^2_{biad}([0,T],L^2(M))$, the functional calculus $\frac{\alpha f}{\alpha+f}\in L^\infty_{biad}([0,T],M)$ clearly converges to $f$ in $L^2$ when $\alpha\to \infty$, giving the GNS construction for $L^\infty_{biad}([0,T],M)$.

The Dirichlet form statement is now obvious since functional calculus is computed pointwise in time. The statement about $\Delta^1$ follows similarly as the previous point (4) once noted the density result in lemma \ref{VoiculescuFormula} of $\alpha_{T-s}(A).$
\end{enumerate}
 \end{proof}

\subsection{Reversed Markovian Stochastic Integrals and SDE for the reversed process}

From the result of (5) in the previous lemma, one can expect that for any $Y\in D(\Delta)$: $$\int_u^v \overline{\delta}_s(Y_s)\#d\overline{S}_s +\int_{T-v}^{T-u}\overline{\delta}_s(Y_s)\#d{S}_s=\int_u^v\Delta_s(Y_s)ds.$$

However, obtaining this result is not obvious and we can't use directly a mere density argument. We start by proving  what we can expect from this statement that $\int_{T-v}^{T-u}\overline{\partial}_{T-s}(Y_{T-s})\#d{S}_s-\int_u^v\Delta_s(Y_s)ds$
is an $\overline{\mathcal{F}}$-martingale with the right quadratic variation, following the proof of our construction of the reversed Brownian motion. This will the key for obtaining the SDE for the reversed process, from the one from the original one.

\begin{lemma}\label{RevStoInt}Assume Assumption (C) and fix $U\in D(\delta_{[0,T]}^*).$
\begin{enumerate}

\item Let $Z_{v}=\int_{T-v}^{T}U_{T-s}\#d{S}_s-\int_0^v\delta_s^*(U_s)ds, T-v\in[0,T]$ is a $\overline{\mathcal{F}}$-martingale.

\item We have the equality : $||Z_{v}||_2^2=\int_{0}^{v}||U_s||_2^2ds$
 \item $$\int_u^v U_s\#d\overline{S}_s +\int_{T-v}^{T-u}U_{T-s}\#d{S}_s-\int_u^v\delta_s^*(U_s)ds=0.$$
 Especially, for any $Y\in D(\Delta)$ : $$\int_u^v \overline{\delta}_s(Y_s)\#d\overline{S}_s +\int_{T-v}^{T-u}\overline{\delta}_{T-s}(Y_{T-s})\#d{S}_s-\int_u^v\Delta_s(Y_s)ds=0.$$
 Moreover for any $Y\in D(\Delta^1)$ we also have : $$\int_u^v \overline{\delta}_s(Y_s)\#d\overline{S}_s +\int_{T-v}^{T-u}\overline{\delta}_{T-s}(Y_{T-s})\#d{S}_s-\int_u^v\Delta_s^1(Y_s)ds=0.$$
\end{enumerate}
\end{lemma}

\begin{proof}
\begin{enumerate}
\item The proof is really similar to the one proving markovianity of $\overline{S}_t$, it thus uses (C.6) crucially. Using free markovianity in assumption (B.2),   $\mathcal{F}_{T-s}$ and $\overline{\mathcal{F}}_s$ are free with amalgamation over $\mathcal{F}_{T-s}\cap\overline{\mathcal{F}}_s=W^*(B,\alpha_{T-s}(X), X\in A)$ so that, since $Z_t-Z_s\in \mathcal{F}_{T-s}$ for $t\geq s$, we can compute  $E(Z_t-Z_s|\overline{\mathcal{F}}_s)=E[\int_{T-t}^{T-s}U_{T-s}\#d{S}_s-\int_s^t\delta_s^*(U_s)ds|W^*(\alpha_{T-s}(A))]$. 

By Ito formula in the form of lemma \ref{ItoSimple}, if we consider the process $\hat{Z}_{T-s}=Z_s-Z_t$, for $T-s\geq T-t$, one can compute, using equation (\ref{ItoPDE}) for any $X\in A$:
\begin{align*}&\hat{Z}_{T-s}\alpha_{T-s}(X)=-\int_{T-t}^{T-s}U_{T-u}\#d{S}_uX_u^{T-s}-\hat{Z}_{T-s}\alpha_{T-u}\otimes\alpha_{T-u}(\delta(K_{u}^{T-s}(X)))\#dS_{u}\\&+\int_{T-t}^{T-s}du[
\delta^*_u(U_u)X_u^{T-s}-m\circ(1\o E_B\circ m\o 1)(U_u\o(\alpha_u\otimes\alpha_u\delta(K_{u}^{T-s}(X)))]\end{align*}


Taking the trace and using the definition of $\delta^*_u(Y_u)$ (and since $\delta_u$ is a real derivation) 
 one deduces :
$$\tau(\hat{Z}_{T-s}\alpha_{T-s}(X))=0.$$
One deduces 
$E(Z_t-Z_s|\overline{\mathcal{F}}_s)=0$ as expected.

\item From the martingale property, one deduces as usual :
$$||Z_{v}||_2^2=\sum_{k=1}^m||Z_{v(k+1)/m}-Z_{vk/m}||_2^2.$$

We want to bound for $v\geq u$, \begin{align*}&\left| ||Z_{v}-Z_{u}||_2^2-\int_{T-v}^{T-u}ds||U_{T-s}||_2^2\right|\\&=\left|-2\Re \langle\int_{T-v}^{T-u} U_{T-s}\#d{S}_s,\int_u^v\delta_s^*(U_s)ds\rangle+||\int_u^v\delta_s^*(U_s)ds||_2^2\right|\\& \leq(v-u)\int_u^v||\delta_s^*(U_s)||_2^2ds+\sqrt{v-u}\left(\int_u^v||\delta_s^*(U_s)||_2^2ds+\int_{T-v}^{T-u}ds||U_{T-s}||_2^2\right)
\end{align*}
We only used Cauchy-Schwarz in the last line. Summing up one gets :

$$\left|||Z_{v}||_2^2-\int_{0}^{v}||U_{s}||_2^2ds\right|\leq (\frac{1}{\sqrt{m}}+\frac{1}{{m}})\int_0^v||\delta_s^*(U_s)||_2^2ds+\frac{1}{\sqrt{m}}\int_{T-v}^{T}ds||U_{T-s}||_2^2\to_{m\to \infty} 0$$

\item  To prove the vanishing of the statement, using (2) we only have to check : 

$$\langle\int_u^v U_s\#d\overline{S}_s ,\int_{T-v}^{T-u}U_{T-s}\#d{S}_s-\int_u^v\delta_s^*(U_s)ds \rangle =-\int_{u}^{v}||U_s||_2^2ds$$



From (6) of lemma \ref{Unbounded} there exists $U^{(k)}$ on which we can apply (5) in this same lemma and converging in $L^2_{biad}$ to $U$. By (5) we thus have $$\int_0^v U_s^{(k)}\#d\overline{S}_s +\int_{T-v}^{T}U_{T-s}^{(k)}\#d{S}_s=\int_0^v\delta_s^*(U_s^{(k)})ds.$$

so that using the polarized version of the result in (2), one gets :
\begin{align*}\langle&\int_u^v U_s^{(k)}\#d\overline{S}_s ,\int_{T-v}^{T-u}U_{T-s}\#d{S}_s-\int_u^v\delta_s^*(U_s)ds \rangle 
\\&\langle-\int_{T-v}^{T}U_{T-s}^{(k)}\#d{S}_s+\int_0^v\delta_s^*(U_s^{(k)})ds ,\int_{T-v}^{T-u}U_{T-s}\#d{S}_s-\int_u^v\delta_s^*(U_s)ds \rangle 
\\ &=-\int_{u}^{v}\langle U_{s}^{(k)},U_{s}\rangle ds.\end{align*}

Taking the limit $k\to \infty$, this concludes. (The reader will note this argument indeed only uses the density result in (6) of lemma \ref{Unbounded} and not a stronger core property for $\delta_{[0,T]}^*$)
The second statement for $Y\in D(\Delta)$ is an immediate consequence using lemma \ref{Unbounded} (4). For $Y\in D(\Delta^1)$, a general Dirichlet form result implies $Y_{\alpha}=\frac{\alpha}{\alpha+\Delta}Y\in D(\Delta), $ 
$\Delta Y_\alpha=\frac{\alpha}{\alpha+\Delta}\Delta^1(Y)$ where we used the canonical $L^1$ extension of $\frac{\alpha}{\alpha+\Delta}$ so that we get a domination $||\Delta Y_\alpha||_1\leq ||\Delta^1(Y)||_1.$

Obviously, we have $\int_u^v \overline{\delta}_s(Y_{\alpha,s})\#d\overline{S}_s +\int_{T-v}^{T-u}\overline{\delta}_{T-s}(Y_{\alpha,T-s})\#d{S}_s\to \int_u^v \overline{\delta}_s(Y_s)\#d\overline{S}_s +\int_{T-v}^{T-u}\overline{\delta}_{T-s}(Y_{T-s})\#d{S}_s$ in $L^2(M)$.

Since again by lemma \ref{Unbounded} (4) $Y_{\alpha,s}=\frac{\alpha}{\alpha+\Delta_s}Y_s$ almost everywhere, one gets a.e. $\Delta_s Y_{\alpha,s}\to \Delta^1_sY_s$ weakly in $L^1(W^*(\alpha_{T-s}(A)))$ thus in $L^1(M).$ Thus, since by the characterization of $D(\Delta^1)$ in lemma \ref{Unbounded} (7) $\int_u^v\Delta_s^1(Y_s)ds$ clearly exists, for any $x\in M$, by dominated convergence Theorem, $\tau(\int_u^v\Delta_s(Y_{\alpha,s})dsx)\to \tau(\int_u^v\Delta_s^1(Y_{s})dsx).$ Gathering everything, this concludes to the last equality.

\end{enumerate} 
\end{proof}
We can now easily deduce the SDE satisfied by the reversed process :

\begin{proposition}\label{RevSDE}
For any $X\in A$, $\overline{\alpha}_t(X):=\alpha_{T-t}(X)\in D(\Delta_s)$ and moreover, $\overline{\alpha}(X)\in D(\Delta^1)$  and $\overline{\alpha}_t(X)$ satisfies the SDE :
$$\overline{\alpha}_t(X)=\overline{\alpha}_0(X)-\int_0^tds[\overline{\alpha}_s(\Delta_{Q,s}(X))+\Delta_s\overline{\alpha}_s(X)]+\int_0^t\overline{\alpha}_s\otimes\overline{\alpha}_s(\delta(X))\#d\overline{S}_s.$$
Moreover, $\overline{\alpha}_t$ satisfies assumption $(A)$ with filtration $\overline{\mathcal{F}}$, Brownian motion $\overline{S}_s$, with $q$ arbitrary in $[1,\infty)$ and drift $\overline{Q}_{i,s}$ given by : $\overline{\alpha}_s(\overline{Q}_{i,s})=-\overline{\xi}_{i,s}-\overline{\alpha}_s(Q_i).$

\end{proposition}
\begin{proof}
 We first check that $\overline{\alpha}_s(X)\in D(\Delta_s)$ for $X\in A,$ using lemma \ref{VoiculescuFormula} and moreover using assumption (C.2):
 
 \begin{align*}&||\Delta_s(\overline{\alpha}_s(X))||\leq \sum_i||(\overline{\alpha}_s\o \overline{\alpha}_s)\delta_i((X))||_{W^*(\overline{\alpha}_s(A))\hat{\o}W^*(\overline{\alpha}_s(A))}\frac{C}{\sqrt{T-s}}\\&+\sum_i||(\overline{\alpha}_s\o \overline{\alpha}_s\o \overline{\alpha}_s)\delta_i\o 1\delta_i((X))||_{W^*(\overline{\alpha}_s(A))\hat{\o}W^*(\overline{\alpha}_s(A))\hat{\o}W^*(\overline{\alpha}_s(A))}\\&+\sum_i||(\overline{\alpha}_s\o \overline{\alpha}_s\o \overline{\alpha}_s)1\o\delta_i\delta_i((X))||_{W^*(\overline{\alpha}_s(A))\hat{\o}W^*(\overline{\alpha}_s(A))\hat{\o}W^*(\overline{\alpha}_s(A))}.\end{align*}
From this and lemma \ref{Unbounded} (7), $\overline{\alpha}(X)\in D(\Delta^1)$ follows immediately. The reader should note that we don't have in general $\overline{\alpha}(X)\in D(\Delta)$ with the bound assumed in (C.2).
   
Using proposition \ref{RevStoInt} (3) one gets (with the integral converging a priori only in $L^1$)  for any $X\in A$~: $$\int_u^v \overline{\delta}_s(\overline{\alpha}_{s}(X))\#d\overline{S}_s =-\int_{T-v}^{T-u}\overline{\delta}_{T-s}(\alpha_{s}(X))\#d{S}_s+\int_u^v\Delta_s(\alpha_{T-s}(X))ds.$$

Thus, replacing this relation in the original SDE in (A.4), we got the expected SDE for $\overline{\alpha}_t(X):=\alpha_{T-t}(X).$
Note  that since by definition $\overline{\alpha}_s$ extends to a von Neumann algebra isomorphism to $L^\infty(A,\tau\circ\overline{\alpha}_s)$, $\overline{Q}$ is uniquely defined by the statement in the proposition. The fact $q$ can be taken arbitrary comes from assumption (C.2) used in our last inequality.

Now the statement about assumption (A) for the reversed process is obvious once noted using Voiculescu's formula in lemma \ref{VoiculescuFormula} again : 
$$\overline{\alpha}_s(\Delta_{\overline{Q},s})(X)=-[\overline{\alpha}_s(\Delta_{Q,s}(X))+\Delta_s\overline{\alpha}_s(X)]$$


\end{proof}

\section{Applications to free Brownian motion and liberation process} 
 
 \subsection{Verifications of Hypothesis in examples}\label{ApplicationSec}

We give interesting basic examples known to satisfy Assumption (C)

\begin{proposition}\label{BrownianC}
Let $X_t=X+S_t$ for $S_t$ a $B$-free Brownian motion of covariance $\eta$, free with amalgamation over $B,$ assumed to have separable predual, with $X=(X_1...,X_n)$ algebraically independent over $B$. Then $X_t$ considered as in example \ref{BrownianDiffusion} satisfy assumption (C).
\end{proposition}
\begin{proof}
The check of assumption (B) is obvious in the setting of example \ref{BrownianDiffusion} with $A_i=B\langle X_i\rangle$, $p=n$ the free Brownian motions those of the proposition, the homomorphisms $\alpha_t$ satisfy $\alpha_t(X)=X_t$, $\delta$ in assumption (A).(1) is the free difference quotient, $\tilde{\delta}_s=0$ in (A).(3) so that (B).(1) is obvious and (A).(4) is given by Ito formula (proposition \ref{ItoSimple}) for the trivial process $X_t$. By definition, we have the strong solution assumption in (B).(2). Since $B$ has separable predual, one can take $\mathscr{B}$ a weak-* dense countable subalgebra so that (0) is now obvious with $\mathscr{A}=\mathscr{B}\langle X_1,...,X_n\rangle.$ (C).(4) is thus also proved in proposition \ref{StrongSolution}.

From the results recalled in subsection \ref{Entropy}, (C).(3) is true with $D=0$, $\overline{\xi}_{i,s}=\xi_i(X_{T-s};B,\eta)=E_{L^2(B\langle X_{T-s}\rangle)}[\frac{S_{T-s}^{(i)}}{T-s}].$ so that since $||S_{t}^{(i)}||\leq 2\sqrt{||\eta(1)||t }$, (C).(2) is valid with $C=2\sqrt{||\eta(1)||}$.

The left continuity in (C).(1) is known from \cite[lemma 4.8 and corollary 4.9]{ShlyFreeAmalg00}. Since we also have $\overline{\xi}_{i,s}=E_{\overline{\mathcal{F}}_s}(\overline{\xi}_{i,t})$ $t\geq s$, the left limit at $s$ is only $E_{\cap_{s<u<t}\overline{\mathcal{F}}_u}(\overline{\xi}_{i,t})$ and is known to exists in $L^2$ for elementary orthogonality reasons. Since $||\overline{\xi}_{i,s}||_2$ is increasing, it has at most countably many points of discontinuity, and each point of discontinuity of $\overline{\xi}_{i,s}$ in $L^2$ should generate a discontinuity of the norm (since the left limit is a projection of the right limit), proving the last statement in (C).(1).

(C).(5) is a consequence of \cite[Proposition 3.8]{ShlyFreeAmalg00} and freeness with amalgamation which implies that the densely (from the strong solution property) defined derivation $\hat{\delta_i}$ of (C).(5) satisfy $\hat{\delta_{si}}^*1\o 1=\delta_{si}*1\o 1.$ The closability then follows from lemma \ref{VoiculescuFormula} and from which we also see $\hat\Delta_s(\alpha_{T-s}(A))\subset L^2(\alpha_{T-s}(A))$ and the extension property for $\hat\Delta_s$ follows.
From the definition here we see $\hat{\delta}_s=\hat{\delta}_t$ on the smallest domain, explaining the end of (C).(5) with $C=1.$
Since for $Q=0$, $\Delta_{Q,t},\Delta_{Q,T-t}$ decrease the degree of the polynomial $P\in A$ strictly, building the solutions in (C).(6) is elementary, for $P$ of degree $n$, one defines by induction $K_t^{s,0}(P)=P$ and for $p\geq 1:$ $$K_t^{s,p}(P)=P+\int_t^s\Delta_{Q,u}(K_u^{s,p-1}(P))du.=P+\sum_{k=1}^p
\int_t^sdu_1\int_{u_1}^sdu_2...\int_{u_{k-1}}^sdu_k \Delta_{Q,u_1}...\Delta_{Q,u_k}(P).$$

and then $K_t^s(P)=K_t^{s,n}(P)=K_t^{s,n+k}(P), k\geq 0.$

From the case $P$ monomial, $(u_1,...,u_k)\mapsto \Delta_{Q,u_1}...\Delta_{Q,u_k}(P)$ is valued in $A$ and continuous with value in $C^2(A_1,...,A_n :B).$ Approximating it by piecewise polynomial functions in $u$, one easily checks $K_t^{s,p}(P)$ belongs to the stated completion $C^1([0,s],C^2(A_1,...,A_n :B))$ and from the inductive relation above, it satisfies 
$$\frac{\partial K_t^s(P)}{\partial t}+\Delta_{Q,t}(K_t^s(P))=0.$$
The construction of $ L_t^s(P)$ is similar.

\end{proof}
More general examples from example \ref{BrownianDiffusion} with $Q_i\neq 0$ polynomial, will be treated elsewhere. We consider here the liberation process.

\begin{proposition}\label{liberationC}
Let $\alpha_t$ the liberation process as in example \ref{liberation} with $Q=0$, $B=\C$ for $S_t$ a free Brownian motion (of covariance $\eta=\tau$), with each $B_i$ having separable predual. Then $\alpha_t$ satisfy assumption (C).
\end{proposition}
\begin{proof}
The check of assumption (B) is obvious in the setting of example \ref{liberation} with $A_i=B_i$, $p=n$ the free Brownian motions those of the proposition,  $\tilde{\delta}_s=0$ in (A).(3) so that (B).(1) is obvious and (A).(4) is given by Ito formula (proposition \ref{ItoSimple}) for the  process $\alpha_t(X),X\in A_i$. By definition, we have the strong solution assumption in (B.2). Since $B,B_i$ have separable predual, one can take ,$\mathscr{B}_i$ a weak-* dense countable $\Q$-subalgebra so that (0) is now obvious with $\mathscr{A}=Alg(\mathscr{B}_1,...,\mathscr{B}_n).$ (C.4) is thus also proved in proposition \ref{StrongSolution}.

From the results recalled in subsection \ref{Entropy} equation \eqref{projLiberation}, (C.3) is true with $D=(\sqrt{T}+4\sqrt{2})C,\alpha=1/2$ with C the constant in (C.2). Indeed $\overline{\xi}_{i,s}=\sqrt{-1}j(U_{T-s}^{(i)}A_iU_{T-s}^{(i)*}:U_{T-s}^{(1)}A_1U_{T-s}^{(1)*}\vee ...  \widehat{U_{T-s}^{(i)}A_iU_{T-s}^{(i)*}}...\vee U_{T-s}^{(n)}A_nU_{T-s}^{(n)*}; B),$ 
and for $t<s$ $$||\overline{\xi}_{i,t}-E_t(\overline{\xi}_{i,s})||_2\leq 2 ||U_{T-t}U_{T-s}^*-1|| ||\overline{\xi}_{i,s}||_2,$$
and we can bound by free Burkholder-Gundy inequality \cite[Th 3.2.1]{BS98}  $U_{T-t}U_{T-s}^*-1=\frac{1}{2}\int_{T-s}^{T-t}U_{T-u}U_{T-s}^*du+i\int_{T-s}^{T-t}dS_uU_{T-u}U_{T-s}^*$ and (C.2) which is a consequence of \cite[Remark 8.10 and corollary 8.3]{Vo6}.

(C.1) is explained in proposition \ref{projLiberation}.(2) and (C.5) is a consequence of the proof there with $C\geq 1.$ Especially for any $s<t$ \begin{equation}\label{hatdeltaLiberation}\forall P\in D(\hat{\delta}_s), \ \ \ \ \  \hat{\delta}_{t,i}(P)=\hat{\delta}_{s,i}(P)\#(U_{T-s}^{(i)}U_{T-t}^{(i)*}\o U_{T-t}^{(i)}U_{T-s}^{(i)*}\end{equation}

Since for $Q=0$, $\Delta_{Q,t},\Delta_{Q,T-t}$ don't increase the degree of the polynomial $P\in A$, building the solutions in (C.6) is elementary. More precisely, for $P$ of degree $n$ (in the grading of the algebra with degree 1 for $B_i$, 0 for $B$), $\Delta_{Q,t}(P)+nP$ is of degree strictly smaller than $deg(P)$ as an elementary computation shows. Let $N$ the multiplication operator defined on homogeneous terms by $NP=deg(P)P$, one defines by induction $K_t^{s,0}(P)=e^{-n(s-t)}P$ and for $p\geq 1:$ \begin{align*} &K_t^{s,p}(P)=e^{-N(s-t)}P+\int_t^se^{-N(u-t)}(\Delta_{Q,u}+N)(K_u^{s,p-1}(P))du.\\&=e^{-N(s-t)}P+\sum_{k=1}^p\int_t^sdu_1\int_{u_1}^sdu_2...\int_{u_{k-1}}^sdu_k e^{-N(u_1-t)}\Delta_{Q,u_1}e^{-N(u_2-u_1)}...\Delta_{Q,u_k}(e^{-N(s-u_n)}P),\end{align*}

and then $K_t^s(P)=K_t^{s,n}(P)=K_t^{s,n+k}(P), k\geq 0.$

The conclusion and construction of $ L_t^s(P)$ is similar to our previous proposition.

\end{proof}
 
\subsection{Alternative formulas for the reversed process and regularity of some martingales}

We got a reversed SDEs, we now want to get alternative formulas and obtain applications to regularity of conjugate variables along free Brownian motions and of liberation gradient along liberation process.
\begin{proposition}\label{AltForm}Assume assumption (C).
\begin{enumerate}
\item For any $X\in A$ 
let us write $R_X(u,t,\overline{X}_u)=\alpha_{T-u}(L_u^t(X)),$ then :
$$\overline{X}_{t}=\int_u^t\delta_{v}(R_X(v,t,\overline{X}_v))\#d\overline{S}_{v}+R_X(u,t,\overline{X}_u)-\int_u^tdv\Delta_v(R_X(v,t,\overline{X}_v)).$$

\item 
Let us write $Q_X(u,t)=E_u(\overline{X}_{t}),$ then 
$$Q_X(u,t)=R_X(u,t,\overline{X}_u)-\int_u^tdvE_u[\Delta_v(R_X(v,t,\overline{X}_v))]).$$
For any $U\in D(\delta^*_{[0,T]})$, $a,b\in A$, we have for $u\leq t\leq T$:
 
\begin{align*} \tau&([\overline{X}_t-Q_X(u,t)]\overline{a}_u\int_u^TU_s\#d\overline{S}_{s}\overline{b}_u)
= \int_u^tdv\tau(Q_X(v,t)\hat{\delta}_v^*(\overline{a}_uU_v\overline{b}_u)).\end{align*}

As a consequence, for any $t\in [u,T]$, $v\mapsto 1_{[u,t)}(v) Q_X(v,t)$ and for any $v\in[u, t]$, $s\mapsto 1_{[u,v)}(s) (E_s\Delta_v(R_X(v,t,\overline{X}_v))$ are in $D(\mathcal{E}),$
and
for any $Z\in D(\mathcal{E})$, we have :
 \begin{align*} \tau&((\overline{X}_t-Q_X(u,t)-\int_u^t\overline{\delta}_s(Q_X(s,t))\#d\overline{S}_{s})\overline{a}_u\int_u^T\overline{\delta}_s(Z_s)\#d\overline{S}_{s}\overline{b}_u)=0.\\ \tau&([(1-E_u)(\Delta_v(R_X(v,t,\overline{X}_v)))-\int_u^v\overline{\delta}_s(E_s\Delta_v(R_X(v,t,\overline{X}_v))\#d\overline{S}_{s})]\overline{a}_u\int_u^T\overline{\delta}_s(Z_s)\#d\overline{S}_{s})\overline{b}_u)=0.\end{align*}

More generally, the result $v\mapsto 1_{[u,t)}(v) Q_X(v,t)\in D(\mathcal{E})$ extend to $X\in L^2(A,\tau_{T-t}), \tau_t=\tau\circ\alpha_t$ and for any  $U\in \mathcal{B}_2^a(\overline{\mathcal{F}})$, 
 
  \begin{align*} \tau&((\overline{X}_t-Q_X(u,t)-\int_u^t\overline{\delta}_s(Q_X(s,t))\#d\overline{S}_{s})\int_u^TU_s\#d\overline{S}_{s})=0.\end{align*}
 
 \item For $X\in L^\infty(A,\tau_{T-v})$, $[u,v]\ni t\mapsto M_t^{[u,v]}(X)=(Q_X(t,v)-Q_X(u,v)-\int_u^t\overline{\delta}_s(Q_X(s,v))\#d\overline{S}_{s})$ and $t\mapsto \mathcal{N}_t^{[u,v]}(X)=\int_u^t\overline{\delta}_s(Q_X(s,v))\#d\overline{S}_{s})$ are martingales in $L^p$ for all $p\in  [1,\infty[$ with null covariation. Actually, $s\mapsto 1_{[u,v)}(s) \overline{\delta}_s(Q_X(s,v))\in \mathcal{B}_p^a(\overline{\mathcal{F}}).$ Moreover, $M_t^{[u,v]}(X)$ has null covariation with any  ${N}_t^{[u,v]}=\int_u^vU_s\#d\overline{S}_{s}$  for all $U\in \mathcal{B}_2^a(\overline{\mathcal{F}})$. Finally, the covariation converges to zero in any $L^r$ with $r<p$ if $U\in \mathcal{B}_p^a(\overline{\mathcal{F}}),p\geq 2.$
\end{enumerate}
\end{proposition}
The proof of (2) is much easier when $a_u,b_u\in B$, the more general case treated here will be important in the second part of this series of papers, and this is already the case with at least one of them not in $B$ since this case is crucial for the regularity part of (3). This is also of course the basis of the improved orthogonality for any stochastic integral. The zero covariation part of (3) can be seen as still an improvement of the orthogonality in (2).

\begin{proof} 
(1) This is an application of proposition \ref{TimeIto} to the reversed process using assumption (C.6) and the fact proved in proposition \ref{RevSDE} that the reversed process satisfy assumption (A).

\medskip
(2) The relation for $Q$ comes from (1). 

From (1) and this formula for $Q$, on deduces :

$$\overline{X}_{t}-Q_X(u,t)=\sum_i\int_u^t\delta_{i,v}(R_X(v,t,\overline{X}_v))\#d\overline{S}_{i,v}-\int_u^tdv(1-E_u)[\Delta_v(R_X(v,t,\overline{X}_v))].$$

We have thus mainly to compute $$\tau((1-E_u)[\Delta_v(R_P(v,t,\overline{X}_v))]\overline{a}_u\int_u^TU_s\#d\overline{S}_{s})\overline{b}_u)=
\tau([\Delta_v(R_P(v,t,\overline{X}_v))]\overline{a}_u\int_u^vU_s\#d\overline{S}_{s})\overline{b}_u)$$
since the stochastic integral is an $\overline{\mathcal{F}}$ martingale. (Note we know from lemma \ref{TimeIto} that $\Delta_v(R_P(v,t,\overline{X}_v))$ is integrable in $L^q$ for any $q<\infty$ since $q$ is arbitrary in assumption (A) for the reversed process.)

Arguing as in our proof of equation \eqref{ItoPDE}, we know that for $v\geq u$, $E_v(\overline{a}_u)=\alpha_{T-v}(K_{T-v}^{T-u}(a))$ and
 \begin{equation}\label{forwardA}\overline{a}_u=E_v(\overline{a}_u)+\int_{T-v}^{T-u}\overline{\delta}_{T-s}E_{T-s}(\overline{a}_u)\#dS_s\end{equation} and note for $v\geq w\geq u$, $E_{w}(\overline{a}_u)=E_v(\overline{a}_u)+\int_{T-v}^{T-w}\overline{\delta}_{T-s}E_{T-s}(\overline{a}_u)\#dS_s$ and similarly for $b$. Also note that from the assumption on $K$, the integrand in the stochastic integral is in $\mathcal{B}_\infty^a.$

From Ito formula in the form of proposition \ref{ItoSimple}, if we write $Z_{v,w}=\int_{T-v}^{T-w}U_{T-s}\#d{S}_s$, one deduces : 

\begin{align*}&\overline{a}_uZ_{v,u}\overline{b}_u=\int_{T-v}^{T-u}[\overline{\delta}_{T-s}E_{T-s}(\overline{a}_u)Z_{v,T-s}E_{T-s}(\overline{b}_u)+E_{T-s}(\overline{a}_u)Z_{v,T-s}\overline{\delta}_{T-s}E_{T-s}(\overline{b}_u))]\#d{S}_s\\&+\int_{T-v}^{T-u}[E_{T-s}(\overline{a}_u)U_{T-s}E_{T-s}(\overline{b}_u)]\#d{S}_s\\ &+\int_{T-v}^{T-u}dsm\circ(1\otimes \eta E_B\circ m\otimes 1)[\delta_{T-s}(E_{T-s}(\overline{a}_u))U_{T-s}E_{T-s}(\overline{b}_u)]\\ &+\int_{T-v}^{T-u}dsm\circ(1\otimes \eta E_B\circ m\otimes 1)[E_{T-s}(\overline{a}_u)U_{T-s}\delta_{T-s}(E_{T-s}(\overline{b}_u))]
\\ &+\int_{T-v}^{T-u}dsm\circ(1\otimes \eta E_B\circ m\otimes 1)[\delta_{T-s}(E_{T-s}(\overline{a}_u))Z_{v,T-s}\overline{\delta}_{T-s}(E_{T-s}(\overline{b}_u))]
\end{align*}



Since the stochastic integrals are  ${\mathcal{F}}$ martingales and $Z_{v,u}$ is orthogonal to ${\mathcal{F}}_v$, one deduces similarly using  also lemma \ref{RevStoInt}(3) $Z_{v,u}=\int_u^v\Delta_s(Y_s)ds-\int_u^v U_s\#d\overline{S}_s .$ (and free markovianity for making appear conditional expectations) and our previous computations:
\begin{align*}\tau&((1-E_u)[\Delta_v(R_P(v,t,\overline{X}_v))]\overline{a}_u(\int_u^TU_s\#d\overline{S}_{s})\overline{b}_u)\\&=
\int_u^vds\tau(E_s[\Delta_v(R_P(v,t,\overline{X}_v))]\overline{a}_u(\Delta_s(Y_s))\overline{b}_u)
\\& -\tau(\int_{u}^vdsE_s[\Delta_v(R_P(v,t,\overline{X}_v))]m\circ(1\otimes \eta E_B\otimes 1)[\delta_{s}(E_{s}(\overline{a}_u))U_sE_{s}(\overline{b}_u)+E_{s}(\overline{a}_u)U_s\delta_{s}(E_{s}(\overline{b}_u))])
\\ &-\tau(\int_{u}^{v}dw\int_w^vdsE_s[\Delta_v(R_P(v,t,\overline{X}_v))]m\circ(1\otimes \eta E_B\otimes 1)[\delta_{w}(E_{w}(\overline{a}_u))\Delta_s(Y_s)\overline{\delta}_{w}(E_{w}(\overline{b}_u))]).\end{align*}
Note we used in the last line $E_s$ where we could have used $E_w$ since $w\leq s\leq v$ so that by free markovianity $E_w|_{\alpha_{T-v}(A)}=E_wE_s|_{\alpha_{T-v}(A)}.$
Let us apply $\hat\delta$ to equation (\ref{forwardA}) for $(s\geq u)$ : $\hat{\delta}_{s}(\overline{a}_u)=\delta_{s}(E_s(\overline{a}_u))+\int_{T-s}^{T-u}\hat{\delta}_{s}
\otimes 1\overline{\delta}_{T-w}E_{T-w}(\overline{a}_u)\#_2dS_w+\int_{T-s}^{T-u}1\otimes\hat{\delta}_{s}
\overline{\delta}_{T-w}E_{T-w}(\overline{a}_u)\#_1dS_w$, as is easily checked by an extension of derivation properties to stochastic integral, knowing the domain properties of $E_{T-w}(\overline{a}_u)$ from its PDE interpretation and assumption (C.6).

Then we apply Ito formula again (after insensitive ``multiplication" by $U_s$) : \begin{align*}&E_s(m\circ(1\otimes \eta E_B\otimes 1)[\hat{\delta}_{s}(\overline{a}_u)U_s]\overline{b}_u)=m\circ(1\otimes \eta E_B\otimes 1)[\delta_{s}(E_s(\overline{a}_u))U_sE_s(\overline{b}_u))\\&+E_s(\int_{u}^{s}dwm\circ(1\otimes\eta E_B\circ m\otimes1)\circ(1^{\otimes 2}\otimes\eta E_B\circ m\otimes m\otimes1)[1\o\hat{\delta}_{s}\overline{\delta}_{w}E_{w}(\overline{a}_u)\otimes U_s\otimes\delta_{w}(E_{w}(\overline{b}_u))]\end{align*}

Thus, using the definition by adjointness of $\Delta_sY_s=\hat{\delta}_s^*\delta_sY_s$ (extension property required in (C.5)) to make appear the second type of integral of our previous formula using  for $b\in  B$  \begin{align*}\tau(bc_s\hat{\delta}_s^*(d_s\o e_s)f_s)&=\langle \hat{\delta}_s((f_sbc_s)^*),d_s\o e_s\rangle\\&=\langle \hat{\delta}_s(c_s^*)b^*,d_s\o e_sf_s\rangle+\langle b^*\hat{\delta}_s(f_s^*),c_sd_s\o e_s\rangle\\&=\tau(b[m(1\o \eta E_B)(\hat{\delta}_s(c_s)d_s)]e_sf_s)))+\tau (b c_sd_s[m( \eta E_B\o 1)(e_s\hat{\delta}_s(f_s))])\end{align*}
which implies
$$E_B(c_s\hat{\delta}_s^*(d_s\o e_s)f_s)=E_B([m(1\o \eta E_B)(\hat{\delta}_s(c_s)d_s)]e_sf_s)))+ c_sd_s[m( \eta E_B\o 1)(e_s\hat{\delta}_s(f_s))]),$$ one gets using the equation obtained from Ito formula and a symmetric variant :
\begin{align*}\tau&(\int_{u}^vdsE_s[\Delta_v(R_P(v,t,\overline{X}_v))]m\circ(1\otimes \eta E_B\otimes 1)[\delta_{s}(E_{s}(\overline{a}_u))U_sE_{s}(\overline{b}_u)+E_{s}(\overline{a}_u)U_s\delta_{s}(E_{s}(\overline{b}_u))])
\\&+\tau(\int_{u}^{v}dw\int_w^vdsE_s[\Delta_v(R_P(v,t,\overline{X}_v))]m\circ(1\otimes \eta E_B\otimes 1)[\delta_{w}(E_{w}(\overline{a}_u))\Delta_s(Y_s)\overline{\delta}_{w}(E_{w}(\overline{b}_u))])\\&=\int_{u}^vds\tau(E_s[\Delta_v(R_P(v,t,\overline{X}_v))]m\circ(1\otimes \eta E_B\otimes 1)[\hat{\delta}_{s}(\overline{a}_u)U_s\overline{b}_u)+\overline{a}_uU_s\hat{\delta}_{s}(\overline{b}_u)])\end{align*}

Thus finally we got using again Voiculescu's formula in lemma \ref{VoiculescuFormula} for $\hat\delta_s^*$ with $1\o 1$  replaced by $U_s$:
\begin{align*}\tau&((1-E_u)[\Delta_v(R_P(v,t,\overline{X}_v))]\overline{a}_u(\int_u^TU_s\#d\overline{S}_{s})\overline{b}_u)\\&=
\int_u^vds\tau(E_s[\Delta_v(R_P(v,t,\overline{X}_v))]\hat\delta_s^*(\overline{a}_u(\delta_s(Y_s))\overline{b}_u)).\end{align*}




Finally, we got using the isometric property of stochastic integral, Fubini Theorem and the definition of $Q$ :

\begin{align*} \tau&([\overline{X}_t-Q_X(u,t)]\overline{a}_u\int_u^TU_s\#d\overline{S}_{s}\overline{b}_u)
\\ &=\int_u^tds\tau(R_P(s,t,\overline{X}_s)\hat{\delta}_s^*(\overline{a}_u(U_s)\overline{b}_u))ds-\int_u^tdv\int_u^v\tau(E_s(\Delta_v(R_P(v,t,\overline{X}_v)))\hat{\delta}_s^*(\overline{a}_u(U_s)\overline{b}_u)))ds\\&= \int_u^tds\tau(Q_P(s,t)\hat{\delta}_s^*(\overline{a}_u(U_s)\overline{b}_u)).\end{align*}

The regularity statements for $Q_P$ and $E_s\Delta_v(R_P(v,t,\overline{X}_v))$ follow from the following characterization, well-known from unbounded operator theory :  $$D(\mathcal{E})=D(\Delta^{1/2})=\{f\in L^2_{biad}, \exists C>0\forall v\in D(\Delta)\ |\langle f,\Delta(v)\rangle|\leq \mathcal{E}(v)^{1/2}C\}.$$

The two orthogonalities follow from the fact $D(\Delta)$ is a core for $\mathcal{E}$ and the equations we already established. The extension to $X\in L^2(A,\tau_{T-t})$ is obvious from the closability of the form $\mathcal{E}$ and from the inequality following from the above orthogonality :
$$||Q_X(u,t)||_2^2+\int_u^t||\delta_sQ_X(s,t)||^2ds\leq ||\overline{X}_t||_2^2$$

The case of orthogonality with any backward stochastic integral is proved first on elementary ones coming from simple biprocesses, which reduces to the previous orthogonality by free markovianity  and taking $U_s=1\o 1$ before in the proof. The general case is obvious by continuity. The reader should note though that the biadapted case was crucial to obtain before to get the domain property of $Q$ by quadratic form techniques.

\medskip
(3)
 The martingale property in $L^2$ is obvious. We start by proving that the covariation converges weakly in $L^1$ to zero when the mesh of the partition goes to zero.

Consider first $\sigma$ a finite partition of $[u,v]$, a fixed interval. We have to estimate the (column) martingale bracket of $M=M^{[u,v]}(X)$ and $\mathcal{N}=\mathcal{N}^{[u,v]}(X)$, $X\in L^\infty(A,\tau_{T-v})$, ${N}=N^{[u,v]}$ the generic stochastic integral as in the second case of quadratic variation statement of the proposition (including the first one as special case when $N=\mathcal{N}$) : $$[M,N]_{\sigma}=\sum_{t\in\sigma, t\neq u}(M_t-M_{t^{-}})^*(N_t-N_{t^{-}})$$ with as usual $t^{-}$ the time just before $t$ in the partition. Taking $a\in M$ and using in the second line the orthogonality in (2) (with $a_u$ replaced by any $E_u(a)b_v,v\leq u$ by free markovianity), one gets :
\begin{align*}\tau(a^*[M,N]_{\sigma})&=\sum_{t\in\sigma, t\neq u}\tau((M_t-M_{t^{-}})^*(N_t-N_{t^{-}})E_t(a^*)))\\&=\sum_{t\in\sigma, t\neq u}\tau((M_t-M_{t^{-}})^*(N_t-N_{t^{-}})(E_t(a^*)-E_{t^{-}}(a^*)))
\\&=\sum_{t\in\sigma, t\neq u}\tau((Q_t-Q_{t^{-}})^*(N_t-N_{t^{-}})(E_t(a^*)-E_{t^{-}}(a^*)))\\&-\sum_{t\in\sigma, t\neq u}\tau((\mathcal{N}_t-\mathcal{N}_{t^{-}})^*(N_t-N_{t^{-}})(E_t(a^*)-E_{t^{-}}(a^*)))=(I)-(II)
\end{align*}
with $(M_t-M_{t^{-}}):=(Q_t-Q_{t^{-}})-(\mathcal{N}_t-\mathcal{N}_{t^{-}})$, i.e. $Q_t=(Q_X(t,v)-Q_X(u,v)$

Let us define $\mathcal{O}_{t,t^{-}}=\int_{t^{-}}^tP_{t^{-}}(\overline{\delta}_s(Q_X(s,v)))\#d\overline{S}_{s}$ where $P_{u}$ is the projection on $\mathcal{H}(\overline{\mathcal{F}}_u,\eta E_B).$

Note that (computing first by density with simple stochastic integrals, on gets : \begin{align}\begin{split}E_{t-}(\mathcal{O}_{t,t^{-}}^*\mathcal{O}_{t,t^{-}})&=E_{t^{-}}(\int_{t^{-}}^tds\langle P_{t^{-}}(\overline{\delta}_s(Q_X(s,v))),P_{t^{-}}(\overline{\delta}_s(Q_X(s,v)))\rangle)\\&\leq E_{t^{-}}(\int_{t^{-}}^tds\langle (\overline{\delta}_s(Q_X(s,v))),(\overline{\delta}_s(Q_X(s,v)))\rangle)
\\&=E_{t-}((\mathcal{N}_t-\mathcal{N}_{t^{-}})^*(\mathcal{N}_t-\mathcal{N}_{t^{-}}))
\leq E_{t-}((Q_t-Q_{t^{-}})^*(Q_t-Q_{t^{-}}))\in M.\end{split}\label{boundO}\end{align}
where the last equality comes from Ito isometry, and in the last inequality we use again the orthogonality from part (2) to add $M$ and replace $N$ by $M+\mathcal{N}=Q.$ As a consequence; all the terms are in $M$ and not only in $L^1(M).$

Especially $V_s=P_{t^{-}}(\overline{\delta}_s(Q_X(s,v)))1_{[t^{-},t[}$ is in $\mathcal{B}_p^a$ for all $p<\infty$ and thus from proposition \ref{Bpa}, 
$\mathcal{O}_{t,t^{-}}\in L^p(M)$ for $p<\infty.$
We will use this to get extra boundedness in writing $(\mathcal{N}_t-\mathcal{N}_{t^{-}})=\mathcal{O}_{t,t^{-}}+\mathcal{P}_{t,t^{-}}$ with $\mathcal{P}_{t,t^{-}}=\int_{t^{-}}^t(1-P_{t^{-}})(\overline{\delta}_s(Q_X(s,v)))\#d\overline{S}_{s},N_{t,t^{-}}=N_t-N_{t-} $ and let us also write $N_{t,t^{-}}=\lim_n N_{t,t^{-}}^n$ for elementary stochastic integrals, especially in $L^p(M).$

Using Cauchy-Schwarz inequality, on gets :
\begin{align*}|(II)|&\leq 2||a||\left((\sum_{t\in\sigma, t\neq u}||P_{t,t^{-}}||_2^2)^{1/2}(\sum_{t\in\sigma, t\neq u}||N_{t,t^{-}}||_2^2)^{1/2}\right)\\&+\left |\sum_{t\in\sigma, t\neq u}\tau((O_{t,t^{-}})^*(N_{t,t^{-}})(E_t(a^*)-E_{t^{-}}(a^*))) \right|=(II_1)+(II_2).\end{align*}

Note that by Ito isometry  $\sum_{t\in\sigma, t\neq u}||P_{t,t^{-}}||_2^2=\int_{u}^v||(1-P_{m(\sigma,s)})(\overline{\delta}_s(Q_X(s,v)))||_2^2$, with $m(\sigma,s)=\max\{t\in \sigma, t\leq s\}$. Let us see this goes to zero along a sequence (or net) of refining partitions of mesh going to zero. Indeed, along refining partitions, $m(\sigma,s)$ is increasing for each $s$, thus $\sigma\mapsto ||(1-P_{m(\sigma,s)})(\overline{\delta}_s(Q_X(s,v)))||_2^2$ is decreasing, and since if the mesh tends to zero $m(\sigma,s)\to s$ thus 
$P_{m(\sigma,s)}\to P_s$ by left continuity of the reversed filtration, so that monotone convergence theorem concludes. Thus $(II_1)$ goes to zero under the same condition. If the net of partitions is not refining, it suffices to apply dominated convergence theorem, as soon as the mesh goes to zero, the conclusion is the same.

Using Cauchy-Schwarz again and the a priori knowledge that $O_{t,t^{-}}\in L^4$ one gets :

\begin{align*}(II_2)&\leq \left|\sum_{t\in\sigma, t\neq u}\tau((E_t(a)-E_{t^{-}}(a))(E_t(a^*)-E_{t^{-}}(a^*))(O_{t,t^{-}})^*(O_{t,t^{-}})) \right|^{1/2}\left|\sum_{t\in\sigma, t\neq u}||N_{t,t^{-}}||_2^2 \right|^{1/2}\\&
\leq (0)^{1/4}\left|\sum_{t\in\sigma, t\neq u}\tau((O_{t,t^{-}})^*(O_{t,t^{-}})(O_{t,t^{-}})^*(O_{t,t^{-}}))) \right|^{1/4} ||N_v-N_u||_2\end{align*}
so that it will be convenient to introduce $$(III)=\left|\sum_{t\in\sigma, t\neq u}\tau((O_{t,t^{-}})^*(O_{t,t^{-}})(O_{t,t^{-}})^*(O_{t,t^{-}}))) \right|,$$
$$(III_n)=\left|\sum_{t\in\sigma, t\neq u}\tau((N^n_{t,t^{-}})^*(N^n_{t,t^{-}})(N^n_{t,t^{-}})^*(N^n_{t,t^{-}}))) \right|,$$
 $$(0)=\left|\sum_{t\in\sigma, t\neq u}\tau([(E_t(a)-E_{t^{-}}(a))(E_t(a^*)-E_{t^{-}}(a^*))]^2) \right|.$$
Likewise, we get : 
\begin{align*}|(I)|&\leq \left((0)^{1/4}||N_{v,u}-N_{v,u}^n||_2+||a||_2(III_n)^{1/4}\right)\\&\times\left(\sum_{t\in\sigma, t\neq u}\tau((Q_t-Q_{t^{-}})^*(Q_t-Q_{t^{-}})(Q_t-Q_{t^{-}})^*(Q_t-Q_{t^{-}}))) \right)^{1/4}.\end{align*}
Since $\left(\sum_{t\in\sigma, t\neq u}||Q_t-Q_{t^{-}})||_4^4 \right)\leq 2||X_v||^2\left(\sum_{t\in\sigma, t\neq u}\tau((Q_t-Q_{t^{-}})^*(Q_t-Q_{t^{-}}))) \right)=2||X_v||^2||X_v||_2^2$ and likewise $(0)\leq 2||a||^2||a||_2^2$, it suffices to bound $(III)$, $(III_n)$.
They will tend to zero with the mesh of the partition since they come from  regularized continuous martingales. Then we will make tend $n\to \infty,||N_{v,u}-N_{v,u}^n||_2\to 0$. Let us explain the bound for (III).

By Ito formula in the form of proposition \ref{ItoLp}, we can write \begin{align*}(O_{t,t^{-}})^*(O_{t,t^{-}})&=\int_{t^{-}}^t[(O_{s,t^{-}})^*P_{t^{-}}(\overline{\delta}_s(Q_X(s,v)))]\#d\overline{S}_{s}+\int_{t^{-}}^t[[P_{t^{-}}(\overline{\delta}_s(Q_X(s,v)))]^*(O_{s,t^{-}})]\#d\overline{S}_{s}
\\&+\int_{t^{-}}^t\langle P_{t^{-}}(\overline{\delta}_s(Q_X(s,v))),P_{t^{-}}(\overline{\delta}_s(Q_X(s,v)))\rangle ds=A+B+C\end{align*}
The three terms are know to be in $L^2$  and $C\in \overline{\mathcal{F}}_{t-}$ is thus orthogonal to $A$ and $B$.

But note also $B=A^*$ and more is true using Ito isometry, we have $$\tau(A^*B)=\tau(B^2)=\int_{t^{-}}^t\tau[\langle (O_{s,t^{-}})^*P_{t^{-}}(\overline{\delta}_s(Q_X(s,v)))],[P_{t^{-}}(\overline{\delta}_s(Q_X(s,v)))]^*(O_{s,t^{-}})]\rangle] ds=0,$$
using again  that the stochastic integrals $O_{s,t^{-}}$ are orthogonal to $\overline{\mathcal{F}}_{t-}$ so that $A$ and $B$ are also orthogonal. 

Thus let us bound using first Ito isometry and \eqref{boudO} :
\begin{align*}&||A||_2^2=||B||_2^2=\int_{t^{-}}^tds\tau(\langle E_{t^{-}}[(O_{s,t^{-}})(O_{s,t^{-}})^*]P_{t^{-}}(\overline{\delta}_s(Q_X(s,v))),P_{t^{-}}(\overline{\delta}_s(Q_X(s,v)))\rangle) ds
\\&\leq \int_{t^{-}}^tds\tau( \int_{t^{-}}^sduE_{t^{-}}[\langle \overline{\delta}_u(Q_X(u,v))),\overline{\delta}_u(Q_X(u,v))\rangle]\langle\overline{\delta}_s(Q_X(s,v))),\overline{\delta}_s(Q_X(s,v))\rangle),\end{align*}

A similar easy computation gives twice this bound for $||C||_2^2$ so that we get :

\begin{align*}(III)&=4\int_{u}^vds\tau( \int_{m(\sigma,s)}^sduE_{m(\sigma,s)}[\langle \overline{\delta}_u(Q_X(u,v))),\overline{\delta}_u(Q_X(u,v))\rangle]\langle\overline{\delta}_s(Q_X(s,v))),\overline{\delta}_s(Q_X(s,v))\rangle)\end{align*}

Since $||\int_{m(\sigma,s)}^sduE_{m(\sigma,s)}[\langle \overline{\delta}_u(Q_X(u,v))),\overline{\delta}_u(Q_X(u,v))\rangle]||\leq (2||X_u||)^2$ again by \eqref{boundO}, we get a domination, and since $m(\sigma,s)\to s$ when the mesh of $\sigma$ tends to zero, the integral $\int_{m(\sigma,s)}^sduE_{m(\sigma,s)}[\langle \overline{\delta}_u(Q_X(u,v))),\overline{\delta}_u(Q_X(u,v))\rangle]$ converges in $L^1$ thus weak-* in $M$  to zero. Thus by dominated convergence theorem, (III) also tends to zero.

To bound $(III)_n$ if we look at refining partitions, we can look at partitions refining the one of $N_n=\int_u^vU_s^n\#d\overline{S}_s$.
We have thus the same computation as for (III) :

\begin{align*}(III_n)&=4\int_{u}^vds\tau( \int_{m(\sigma,s)}^sduE_{m(\sigma,s)}[\langle U^n_u,U^n_u\rangle]\langle U^n_s,U^n_s)\rangle)\to_{\sigma} 0\end{align*}
This concludes the proof that the covariation converges to zero weakly in $L^1$, and even normwise in $L^1$, since our bound above is uniform in $||a||$. Let us now see that $\mathcal{N}$ is a martingale in $L^p$.

We bound using orthogonality and \eqref{boundO} :

\begin{align*}||&\sum_{t\in\sigma, t\neq u}(O_{t,t^{-}})^*(O_{t,t^{-}})||_2^2=2\tau[\sum_{t\in\sigma, t\neq u}(O_{t,t^{-}})^*(O_{t,t^{-}})E_t(\sum_{s\in \sigma,s>t}(O_{s,s^{-}})^*(O_{s,s^{-}}))]+(III)^2
\\& =2\tau[\sum_{t\in\sigma, t\neq u}(O_{t,t^{-}})^*(O_{t,t^{-}})E_t(\sum_{s\in \sigma,s>t}E_{s^{-}}[(O_{s,s^{-}})^*(O_{s,s^{-}})])]+(III)^2
\\& \leq 2\tau[\sum_{t\in\sigma, t\neq u}(O_{t,t^{-}})^*(O_{t,t^{-}})E_t(\sum_{s\in \sigma,s>t}(Q_s-Q_{s^{-}})^*(Q_s-Q_{s^{-}})]+(III)^2
\\& = 2\tau[\sum_{t\in\sigma, t\neq u}(O_{t,t^{-}})^*(O_{t,t^{-}})E_t((Q_v-Q_{t})^*(Q_v-Q_{t})]+(III)^2
\\&\leq 8||X_v||^2||X_v||_2^2+(III)^2
\end{align*}
Taking an ultrafilter $\mathcal{U}$ on partitions as in \cite{JungePerrin} there is an $L^2$ weak limit $w-L^2 \Lim_{\sigma,\mathcal{U}} \sum_{t\in\sigma, t\neq u}(O_{t,t^{-}})^*(O_{t,t^{-}}).$

Let us deduce there is an $L^1$ weak limit to $[\mathcal{N},\mathcal{N}]_{\sigma}$ along $\mathcal{U}$. Indeed we have : $$[\mathcal{N},\mathcal{N}]_{\sigma}=\sum_{t\in\sigma, t\neq u}(O_{t,t^{-}})^*O_{t,t^{-}}+\sum_{t\in\sigma, t\neq u}(P_{t,t^{-}})^*O_{t,t^{-}}+\sum_{t\in\sigma, t\neq u}((O_{t,t^{-}})^*P_{t,t^{-}}\sum_{t\in\sigma, t\neq u}((P_{t,t^{-}})^*P_{t,t^{-}}$$ and the last three terms tend to zero normwise in $L^1$ by our previous bound so that $w-L^1\Lim_{\sigma,\mathcal{U}}[\mathcal{N},\mathcal{N}]_{\sigma}=w-L^2 \Lim_{\sigma,\mathcal{U}} \sum_{t\in\sigma, t\neq u}(O_{t,t^{-}})^*(O_{t,t^{-}})\in L^2,$ and we have:
$$||w-L^1\Lim_{\sigma,\mathcal{U}}[\mathcal{N},\mathcal{N}]_{\sigma}||_2^2\leq 8||X_v||^2||X_v||_2^2$$

Actually, we are now ready to conclude more since $[\mathcal{N},\mathcal{N}]_{\sigma}+[M,M]_\sigma=[M+\mathcal{N},M+\mathcal{N}]_{\sigma}-[M,\mathcal{N}]_{\sigma}-[\mathcal{N},M]_{\sigma}.$ and we saw the last two terms tend weakly in $L^1$ to zero, since moreover $M+\mathcal{N}$ is bounded in $M$, thus the covariation as a weak $L^p$ limit for $p>1$ as in \cite{JungePerrin} using mainly Pisier-Xu noncommutative Burkholder-Gundy inequalities (cf their equation (2.1) we use bellow). As a consequence, we know the following limits exists and with a bound for $p\in [1,\infty[$ :
\begin{align*}||w-L^1\Lim_{\sigma,\mathcal{U}}[\mathcal{N},\mathcal{N}]_{\sigma}||_p&\leq ||w-L^1\Lim_{\sigma,\mathcal{U}}[\mathcal{N},\mathcal{N}]_{\sigma}+w-L^1\Lim_{\sigma,\mathcal{U}}[M,M]_{\sigma}||_p\\&\leq ||w-L^p\Lim_{\sigma,\mathcal{U}}[M+\mathcal{N},M+\mathcal{N}]_{\sigma}||_p\\&\leq 4\alpha_p^2 ||X_v||_p^2.\end{align*}

It mostly remains to compute $w-L^1\Lim_{\sigma,\mathcal{U}}[\mathcal{N},\mathcal{N}]_{\sigma}$ and get the expected value $$\int_{u}^vds\langle (\overline{\delta}_s(Q_X(s,v))),(\overline{\delta}_s(Q_X(s,v)))\rangle.$$ This will conclude the proof of our statement in showing $\mathcal{N}$ is a stochastic integral from an element in $\mathcal{B}_p^a.$
Thus take $a\in M$ and decompose by orthogonality to bound : \begin{align*}&\left|\tau([\mathcal{N},\mathcal{N}]_{\sigma}a)-\tau(\int_{u}^vds\langle (\overline{\delta}_s(Q_X(s,v))),(\overline{\delta}_s(Q_X(s,v)))\rangle a)\right|\\&\leq \left|\sum_{t\in\sigma}\tau((\mathcal{N}_t-\mathcal{N}_{t^{-}})^*(\mathcal{N}_t-\mathcal{N}_{t^{-}})E_{t^{-}}(a)))-\tau(\int_{u}^vds\langle (\overline{\delta}_s(Q_X(s,v))),(\overline{\delta}_s(Q_X(s,v)))\rangle a)\right|\\&+\left|\sum_{t\in\sigma}\tau((\mathcal{N}_t-\mathcal{N}_{t^{-}})^*(\mathcal{N}_t-\mathcal{N}_{t^{-}})(E_t-E_{t^{-}})(a)))\right|=(IV)+(II).\end{align*}

Since we already bounded (II) is remains to bound (IV), but using an equality proved in \eqref{boundO}, we have :
$$|(IV)|= \left|\tau(\int_{u}^vds\langle (\overline{\delta}_s(Q_X(s,v))),(\overline{\delta}_s(Q_X(s,v)))\rangle (a-E_{m(s,\sigma)}(a)))\right|$$
and the convergence to zero again follows from dominated convergence theorem using the left continuity of the filtration. We thus also actually see the weak limit in $L^1$ of $[N,N]_{\sigma}$ without using an ultrafilter.

Now that we know $M$ is a martingale in $L^p$ for all $p>1$ we can improve the convergence of the quadratic variation with $N$ coming from $U\in B_p^a$. From Hölder inequality for covariations, the covariation is bounded in any  $L^r$, $r< p$, thus by Hölder again, and the normic convergence in $L^1$ above, this concludes.

\end{proof}
\medskip

\subsection{Consequences for regularity of conjugate variables and liberation gradient}

\begin{corollary}\label{BrownianConj}
Let $X_t=(X_1+S_t^1...,X_n+S_t^n)$ be a B-free Brownian motion of covariance $\eta$  starting à $(X_1,...,X_n)$ as in proposition \ref{BrownianC}. Then, for almost every any $s>0$, the conjugate variable $\xi_{i,s}=\xi_i(X_1+S_t^1...,X_n+S_t^n:B,\eta)$ is in the domain of the $L^2$ closure of the corresponding free difference quotient  and for any $s\geq t$:

$$\Phi^*(X_t:B,\eta)\geq \Phi^*(X_s:B,\eta)+\int_s^tdu\sum_{i}||\delta(\xi_{i,u})||_{\mathcal{H}(W^*(X_u);\eta\circ E_B)^n}^2.$$

Moreover $\int_s^t m\circ( 1\o \eta E_B m\o 1)[\delta(\xi_{i,u})\o \delta(\xi_{i,u})]\in L^p(M)$ for any $p<\infty.$
\end{corollary}

\begin{remark} In \cite[section 3.8]{VoS}, the unknown continuity of $t\mapsto \Phi^*(X_t:B,\eta)$ is mentioned as a technical problem. It is still not known to be true (or false), but the inequality above improves the decreasingness of this function, and the inequality above is expected to be an equality, as in the classical case. Especially, as the proof bellow will show, this equality would be true if one could prove that the reversed process is a strong solution, i.e. the reversed filtration $\overline{\mathcal{F}}$ is generated by the reversed free Brownian motion. It would be also enough to prove $M_{X_i}^{[0,T]}=0$ This is the aim of the investigation of further regularity of the reversed process in the following parts of this sequence of papers.
\end{remark}

\begin{proof}
We already recalled $\xi_{i,s}=E_{B\langle W^*(X_s)\rangle}(\xi_{i,v})$ for $s>v$, thus applying Proposition \ref{AltForm} to $X=\xi_{i,v}\in L^\infty(A,\tau_{v})$ one gets 
an improved reversed martingale property of $\xi_{T-s}$ on $[0,T-u[$ as :
\begin{equation}\xi_{i,s}=\xi_{i,T}+\int_0^{T-s}\overline{\delta}(\xi_{T-w}^i)\#d\overline{S}_{w}+dM_{T-s}^{[0,T-u]}(\xi_{i,u})\end{equation}
Given the quadratic variation computed in Proposition \ref{AltForm} (3) our three terms are orthogonal and one deduces the expected  inequality :
$||\xi_{i,s}||_2^2\geq ||\xi_{i,T}||_2^2+\int_T^{s}||\overline{\delta}(\xi_{w}^i)||^2dw.$

\end{proof}

\begin{corollary}\label{LiberationConj}
Let $\alpha_t(A_i)$ be a liberation process starting at $(A_1,...,A_n)$ in presence of $B$ as in proposition \ref{liberationC}. Then, for almost every any $t>0$, the liberation gradient $j_{i,t}=j(U_t^{(i)}A_iU_t^{(i)*};U_t^{(1)}A_1U_t^{(1)*}\vee ... \widehat{U_t^{(i)}A_iU_t^{(i)*}}...\vee U_t^{(n)}A_nU_t^{(n)*}: B)$ is in the domain of the $L^2$ closure of the corresponding liberation derivations $\delta_i$ and for any $s\geq t$ if $\delta_B=\sum_{i}\delta_i$, then :

\begin{align*}\varphi^*&(\alpha_t(A_1);...,\alpha_t(A_n):B)\\&\geq \varphi^*(\alpha_s(A_1);...,\alpha_s(A_n):B)+\int_s^tdu\sum_{i}\left(||(\delta_i-\delta_B)(j_{i,u})||_{2}^2+\sum_{j\neq i}||\delta_j(j_{i,u})||_{2}^2\right).\end{align*}
\end{corollary}

\begin{proof}
From \eqref{projLiberation} and proposition \ref{liberationC} , 
$\overline{\xi}_{i,T-t}=j_{i,t}=j(U_t^{(i)}A_iU_t^{(i)*}:U_t^{(1)}A_1U_t^{(1)*}\vee ... \widehat{U_t^{(i)}A_iU_t^{(i)*}}...\vee U_t^{(n)}A_nU_t^{(n)*}; B)=E_t[U_t^{(i)}U_s^{(i)*}\overline{\xi}_{i,T-s}U_s^{(i)}U_t^{(i)*}]$ and this description, not being a projection of an element of $\alpha_s(A)$, is not well suited for application of proposition \ref{AltForm}. We would prefer to project first and then conjugate. For, consider the homomorphisms on the same algebra as in proposition \ref{AltForm}  and with the same notation (same $\alpha_t,$  same $\delta$ etc.), defined as for $t, s\geq 0$ $\alpha_{t}^{(i)}(X)=U_{s}^i(U_{t+s}^i)^* U_{t+s}^j\alpha_0(a_j)(U_{t+s}^j)^*U_{t+s}^i(U_{s}^i)^*$, if $a_j\in B_j.$ 

Note that $W_{t,j}^{s,i}=U_{s}^i(U_{t+s}^i)^* U_{t+s}^j$ satisfies 
\begin{align*}W_{t,j}^{s,i}&=U_{s}^j-\int _s^{t+s}1_{i\neq j}W_{v,j}^{s,i}dv+i\int_s^{t+s}U_{s}^i(U_{v+s}^i)^*dS_{s+v}^{(j)}U_{v+s}^j-i\int_s^{t+s}U_{s}^i(U_{v+s}^i)^*dS_{s+v}^{(i)}U_{v+s}^j
\\&=U_{s}^j-\int _s^{t+s}1_{i\neq j}W_{v,j}^{s,i}dv+i\int_0^{t}dS_v^{(i,j)}W_{v,j}^{s,i}-i\int_0^{t}dS_{v}^{(i,i)}W_{v,j}^{s,i},\end{align*}
where $S_t^{(i,j)}=\int_s^tU_{s}^i(U_{v+s}^i)^*dS_{s+v}^jU_{v+s}^i(U_{s}^i)^*$ is easily seen to be a free Brownian motion  by Theorem \ref{FreeLevy}.

In this way $\alpha_{0}^{(i)}=\alpha_s$ of proposition \ref{AltForm} and this equality holds for all times on $B_i.$ If we define $\delta^{(i)}=(\delta^{(i)}_1,...,\delta^{(i)}_n)$ by $\delta^{(i)}_j=\delta_j$ if $j\neq i$, $\delta^{(i)}_i=-\sum_{j\neq i}\delta_j=\delta_i-\delta_B$ ($\delta_i$ are the derivation computed on the variables $\alpha_t^i(B_j)$ as in our general setting, $\delta_B(c)=-i(c\o 1-1\o c)$ if $c\in A_i$, $\delta_B(B)=0$, i.e. $\delta_B=\sum_j\delta_j$), then Ito formula shows $\alpha_{t}^{(i)}$ satisfy assumption (B) with as Brownian motion $S_t^{(i,j)}$ above and $Q=0$. It is easy to check that assumption (C) is satisfied. We work on $[0,T-s]$ for the time reversal for consistency with the original time reversal on $[0,T].$

Note that as in \eqref{projLiberation}, we have for $t>s$ $\overline{\xi}_{j,T-s-t}^i=E_{\alpha_{t}^i(A)}[W_{t,j}^{s,i}(U_{s}^j)^*\overline{\xi}_{i,T-s}U_{s}^j(W_{t,j}^{s,i})^*], j\neq i$ and 
$\overline{\xi}_{i,T-s-t}^i=-\sum_{j\neq i}E_{\alpha_{t}^i(A)}[W_{t,j}^{s,i}(U_{s}^j)^*\overline{\xi}_{i,T-s}U_{s}^j(W_{t,j}^{s,i})^*].$
But note that even if it does not appear in the time reversal of this process, since the $i$-th algebra does not change after time $s$, $\overline{\xi}_{i,T-t-s}=U_{t+s}^i(U_{s}^i)^*E_{\alpha_{t}^i(A)}[\overline{\xi}_{i,T-s}](U_{t+s}^i)^*U_{s}^i
=$
since $\alpha_t^i(X)=U_{s}^i(U_{t+s}^i)^* \alpha_{s+t}(X)U_{t+s}^i(U_{s}^i)^*.$

This is convenient for us since we can apply proposition \ref{AltForm} (3) to $E_{\alpha_{t}^i(A)}[\overline{\xi}_{i,T-s}],$ calling $\overline{S_t}^{(i,j)}$ the reversed Brownian motion and $M^{[0,T-s],i}_t(X)$ the extra martingale for $s<s+t<T$:

$E_{\alpha_{t}^i(A)}[\overline{\xi}_{i,T-s}]=E_{\alpha_{T-s}^i(A)}[\overline{\xi}_{i,T-s}]+\sum_j\int_0^t\delta_j^{(i)}(E_{\alpha_{T-s-u}^i(A)}[\overline{\xi}_{i,T-s}])\#d\overline{S_u}^{(i,j)}+M^{[0,T-s],i}_t(\overline{\xi}_{i,T-s}).$

we thus obtain the concluding inequality :

$$||\overline{\xi}_{i,T-t-s}||_2^2\geq ||\overline{\xi}_{i,0}||_2^2+\int_0^{t}||(\delta_i-\delta_B)\overline{\xi}_{i,T-s-u}||_2^2du+\sum_{j\neq i}\int_0^{t}||\delta_j\overline{\xi}_{i,T-s-u}||_2^2du$$

\end{proof}
\subsection{A computational application to liberated projections}
In the case of proposition \ref{liberationC} when $n=2$,$A_1=\C P+\C(1-P)$,$B=\C Q+\C(1-Q)$,  for $P,Q$ two projections with $\tau(P),\tau(Q)\leq 1/2$. One usually writes $\alpha_t(P)=P_t$ and then $QP_tQ=\alpha_t(QPQ)$ is called the operator valued angle of the liberated pair of projections and appears crucially in understanding for instance orbital entropy of the pair of algebras $A_1,B$. Let $r_t=P_t\wedge Q$. In \cite[section 1.4]{CollinsKemp}, the question is raised how to compute the left derivative of $F_T(s)=\tau((r_TP_sr_T-r_T)^2)$ at $s=T$, the right derivative is computed in their section 2.2 with the help of Ito's calculus. The computation of the backward process on $[0,T]$ as a free stochastic integral exactly enables this kind of computation. Recall the forward SDE special case of example \ref{liberation} : 
$$P_t=p+\int_0^t(\tau(P)-P_s)ds+i\int_0^t[dS_s,p_s]$$
Thus let $\overline{P}_t=P_{T-t}$ and $\overline{j}_t=j^*(\C \overline{P}_t+\C(1-\overline{P}_t):\C Q+\C(1-Q))$ so that using Proposition \ref{RevSDE} thanks to proposition \ref{liberationC}, our result states that :
$$\overline{P}_t=\overline{P}_0+\int_0^t(\tau(P)-\overline{P}_s-[\overline{P}_s,\overline{j}_{s} ])ds+i\int_0^t[d\overline{S}_s,\overline{P}_s].$$

(The reader should remember $i\overline{j}_{s}=\overline{\xi}_s$ to make the sign computation right.)

Then one can compute as in \cite[section 2.2]{CollinsKemp}, the derivative of  $\overline{F}_T(s)=F_T(T-s)=\tau(r_T-2r_T\overline{P}_s+r_T\overline{P}_sr_T\overline{P}_s)$ so that at any $s\geq 0$ :

\begin{align*}\frac{d\overline{F}_T}{ds}(s)&= -2[-\tau(r_T\overline{P}_s)+\tau(P)\tau(r_T)+\tau(r_T[\overline{j}_{s} ,\overline{P}_s])]\\&+2[-\tau((r_T\overline{P}_s)^2)+\tau(r_T\overline{P}_s)\tau(P)+\tau((r_T\overline{P}_sr_T[\overline{j}_{s} ,\overline{P}_s])])]\\&+[-2(\tau(r_T\overline{P}_s))^2+2\tau(r_T)\tau(r_T\overline{P}_s)] \\
&= 2\left((1+\tau(P))\tau(r_T\overline{P}_s) + \tau(r_T)[\tau(r_T\overline{P}_s)-\tau(P)] - [\tau((r_T\overline{P}_s)^2)+ (\tau(r_T\overline{P}_s))^2]\right)
\\&+ 2[\tau(r_T[\overline{j}_{s} ,\overline{P}_s])+\tau((r_T\overline{P}_sr_T[\overline{j}_{s} ,\overline{P}_s])]\end{align*}

Apart from a $(1+\tau(P))$ instead of a $2$ in \cite{CollinsKemp} that comes from a typo in their formula (2.27), the new terms are gathered on the last line. But at $s=0$ we have : $$(r_T[\overline{j}_{s} ,\overline{P}_s]r_T)=(r_T(\overline{j}_{0}P_T-P_T\overline{j}_0)r_T)=r_T(\overline{j}_{0}r_T-r_T\overline{j}_0)r_T=0$$ since  $r_Tp_T=r_T=r_Tp_T$ and thus the supplementary term vanish altogether. Thus whatever the derivative is, it is the same as the one for the forward process and in our computations, we have :
\begin{align*}\frac{d\overline{F}_T}{ds}(0)&=  2\left((1+\tau(P))\tau(r_T) + \tau(r_T)[\tau(r_T)-\tau(P)] - [\tau(r_T)+ \tau(r_T)^2]\right)
=0 \end{align*}
Since $\tau(r_T)=0$ anyways by \cite[lemma 12.5]{Vo6}, the computational mistakes are not such important but we find  $\frac{dF_T}{ds}$ is differentiable at $T$ with vanishing derivative.

\end{document}